\documentclass[11pt,a4paper,twoside]{article}

\usepackage{amssymb,amsmath,graphicx,multicol,amsthm,epstopdf,fancyhdr,xcolor,titlesec}
\usepackage{mathrsfs,mathtools,amsfonts}
\usepackage[colorlinks=true, pdfstartview=FitV, linkcolor=blue, citecolor=blue, urlcolor=blue]{hyperref}
\usepackage{cleveref}
\usepackage[english]{babel}
\usepackage[bottom,multiple]{footmisc}
\usepackage[margin=3cm]{geometry}
\usepackage[font=scriptsize,labelfont=bf]{caption,subcaption}

\usepackage{lastpage}

\titleformat{\section}[block]{\bfseries\large}{\thesection. }{2pt}{}
\titleformat{\subsection}[block]
  {\bfseries\itshape}  % Format: normal font, large size, bold italic
  {\thesubsection.}{2pt}{}
\theoremstyle{plain}
\newtheorem{Th}{Theorem}[section]
\newtheorem{lemma}[Th]{Lemma}

\newtheorem{proposition}[Th]{Proposition}

\theoremstyle{definition}
\newtheorem{definition}[Th]{Definition}

\newtheorem{remark}[Th]{Remark}

\newtheorem{Ex}[Th]{Example}

\newtheorem{theorem}[Th]{Theorem}

\allowdisplaybreaks
\everymath{\displaystyle}
\newcommand{\R}{\mathbb{R}}

\newcommand{\C}{\mathbb {C}}

\newcommand{\changetoline}[1]{%
    \let\parallel\|%
    #1
}

\begin{document}
\thispagestyle{empty}
\begin{center}
{\Large \textbf{The Two-Sided Clifford Dunkl Transform and Miyachi's Theorem}}
\end{center}

\vspace{1cm}

\begin{center}
\textbf{Mohamed Essenhajy} \footnote{Corresponding Author: mohamed.essenhajy@edu.umi.ac.ma}\footnote{Faculty of Sciences, Moulay Ismail University of Meknes, Morocco.} and \textbf{Said Fahlaoui}\footnote{saidfahlaoui@gmail.com}\footnote{Faculty of Sciences, Moulay Ismail University of Meknes, Morocco.}
\end{center}

\leftskip=2cm

\vspace{2cm}

\noindent\textbf{Abstract.}
\\\\
In recent developments, significant progress has been made in extending the Dunkl transform within the framework of Clifford algebras. In particular, the two-sided quaternionic Dunkl transform, which serves as a Dunkl analogue of the two-dimensional quaternionic Fourier transform, has been introduced. In this paper, we present the two-sided Clifford Dunkl transform, defined with two square roots of \(-1\) in \(Cl_{p,q}\). We investigate various properties of this transform, including the inverse theorem and the Plancherel formula, and derive two explicit formulas for the corresponding translation operator. Furthermore, we establish an analogue of Miyachi’s theorem for this transform.

\leftskip=0cm 

\vspace*{2cm}

\section{Introduction}

\noindent The Dunkl theory generalizes classical Fourier analysis on the Euclidean space $\mathbb{R}^d$. It was originally introduced by Charles Dunkl in \cite{4444} and has since been developed by many mathematicians, finding applications in various fields of mathematics and mathematical physics (see \cite{key-3}, \cite{DAP1}). In particular, the Dunkl transform has proven to be a powerful tool in harmonic analysis, providing a framework for studying differential-difference operators associated with reflection groups. Moreover, recent works have further advanced and extended Dunkl harmonic analysis to new settings (see \cite{key-BM}, \cite{key-TD3}, \cite{key-TD4}, and \cite{key-TD5}).

\noindent Clifford analysis (see \cite{key-BR}) extends the theory of holomorphic functions from the complex plane to higher dimensions, providing a refined framework for harmonic analysis in Euclidean space. In recent years, significant progress has been made in generalizing classical results from complex and quaternionic analysis to higher dimensions via Clifford algebras. Notable contributions in this direction include \cite{key-TD1}, \cite{key-TD2}, \cite{key-MF1}, and \cite{key-MF2}, which have paved the way for novel applications in signal processing, probability theory, and the establishment of uncertainty principles within the Clifford framework. Overall, these developments underscore the versatile role of Clifford algebras in advancing both theoretical and applied aspects of hypercomplex analysis, with numerous applications in physics and hypercomplex signal theory (see \cite{key-BR2}, \cite{Dor}, \cite{phy}, \cite{GHH}, \cite{Por}, \cite{jv}).

Since 2005, Dunkl theory has attracted growing interest within the Clifford analysis community (see \cite{OPD}, \cite{DBY}, \cite{key-fah}, \cite{key-uncertainty}). In \cite{key-fah}, the authors introduced and studied a Dunkl version of the quaternionic Fourier transform, known as the two-sided quaternionic Dunkl transform. However, a more general framework encompassing Clifford algebras has remained unexplored.

The aim of this paper is twofold. First, we propose a definition of the two-sided Clifford Dunkl transform (CDT), which generalizes the two-sided quaternionic Dunkl transform by incorporating two square roots of $-1$ in the Clifford algebra $Cl_{p,q}$. We establish fundamental properties of the CDT, including the inversion formula, Plancherel identity, and convolution theorem. 

Second, we extend Miyachi’s theorem to the CDT setting using properties of the Dunkl kernel. This result provides an analogue of a well-known theorem in Fourier analysis, further demonstrating the depth of the Clifford Dunkl framework.

%Our approach differs from existing works such as those of Daher and co-authors (see \cite{...}), where alternative generalizations of the Dunkl transform were studied. Unlike their constructions, our definition of CDT retains a two-sided structure, allowing for broader applications in Clifford analysis.

The structure of the paper is as follows: Section 2 provides an overview of Dunkl analysis for readers unfamiliar with the topic. Section 3 reviews fundamental notions of Clifford algebras. In Section 4, we construct the CDT and establish its key properties. Finally, Section 5 is dedicated to the extension of Miyachi’s theorem in this setting.

\section{Preliminaries on Dunkl-Clifford Analysis}

\subsection{Fundamentals of Dunkl Analysis}
This subsection introduces the fundamental concepts of Dunkl theory that will be required in the subsequent sections. For detailed accounts, we refer the reader to \cite{key-3}, \cite{key-5}, \cite{key-6}, \cite{co}, and \cite{MR}.

Consider the Euclidean space $\mathbb{R}^{d}$ equipped with the Euclidean inner product $\langle \cdot, \cdot \rangle$ and norm $\| x \| = \sqrt{ \langle x,x \rangle}$. We extend the inner product $\langle \cdot, \cdot \rangle$ to a bilinear form on $\mathbb{C}^{d} \times \mathbb{C}^{d}$, employing the same notation for this extension.

Let $G$ be a finite reflection group on $\mathbb{R}^{d}$, associated with a root system $R$. A multiplicity function associated with the root system $R$ is a function $\textbf{\textit{k}}_{\scriptscriptstyle d} : R \rightarrow \mathbb{C}$ such that $\textbf{\textit{k}}_{\scriptscriptstyle d}(\alpha) = \textbf{\textit{k}}_{\scriptscriptstyle d}(w(\alpha))$ for all $w$ in $G$. The Dunkl operators $T^{\scriptscriptstyle d}_{j}$, $j=1,\ldots,d$, on $\mathbb{R}^{d}$, associated with the reflection group $G$ and the multiplicity function $\textbf{\textit{k}}_{\scriptscriptstyle d}$, are defined for a function $f$ of class $\mathcal{C}^{1}$ on $\mathbb{R}^{d}$ by
\[
T^{d}_{j} f(x) = \frac{\partial}{\partial x_{j}} f(x) + \sum_{\alpha \in R_{+}} \textbf{\textit{k}}_{\scriptscriptstyle d}(\alpha) \frac{f(x) - f(\sigma_{\alpha}(x))}{\langle x, \alpha \rangle} \alpha_{j},
\]
where $R_{+}$ is a set of positive roots of $R$, and $\sigma_{\alpha}$ denotes the reflection corresponding to a nonzero root $\alpha \in R_{+}$.

Throughout this paper, we assume that $\textbf{\textit{k}}_{\scriptscriptstyle d} \geq 0$.

For a fixed $y \in \mathbb{R}^{d}$, the Dunkl kernel $E_{{d}, \textbf{\textit{k}}_{\scriptscriptstyle d}}(x, y)$ is the unique solution to the joint eigenvalue problem for the Dunkl operators:
\[
\left\{
\begin{array}{ll}
   T^{d}_{j}f(x,y)=y_{j}f(x,y) & \text{for } j=1, \ldots, d; \\
   f(0,y)=1. & 
\end{array}
\right.
\]

This kernel has a unique holomorphic extension to $\mathbb{C}^{d} \times \mathbb{C}^{d}$ and satisfies the following properties:
 \begin{enumerate}
 \item[(i)] $E_{d, \textbf{\textit{k}}_{\scriptscriptstyle d}}(\lambda x,y) = E_{d, \textbf{\textit{k}}_{\scriptscriptstyle d}}(x,\lambda y)$ and $E_{d, \textbf{\textit{k}}_{\scriptscriptstyle d}}(x,y)=E_{d, \textbf{\textit{k}}_{\scriptscriptstyle d}}(y,x)$ for all $x, y \in \mathbb{C}^{d}$ and $\lambda \in \mathbb{C}$.
 \item[(ii)] For all $\nu \in \mathbb{N}^{d}$, $x \in \mathbb{R}^{d}$, and $y \in \mathbb{C}^{d}$, we have
 \begin{align}\label{DK}
 \| D^{\nu}_{y} E_{d, \textbf{\textit{k}}_{\scriptscriptstyle d}}(x,y) \| \leq \| x \|^{\textbf{\textit{l}}( \nu )} e^{ \| x \| \| \operatorname{Re}(y) \|},
 \end{align}
 where $D^{\nu}_{y} = \frac{\partial^{\textbf{\textit{l}}(\nu)}}{( \partial y^{\nu_{1}}_{1} \cdots \partial y^{\nu_{d}}_{d})}$ and $\textbf{\textit{l}}( \nu ) = \nu_{1} + \cdots + \nu_{d}$.\\ 
 In particular, for all $x, y \in \mathbb{R}^{d}$, we have
 \begin{align}\label{yy}
  \| E_{d, \textbf{\textit{k}}_{\scriptscriptstyle d}}(-ix,y) \| \leq 1,
\end{align}
where $i$ denotes the imaginary unit.
 \end{enumerate}
We introduce the index $\gamma$ of the root system as
\[ 
\gamma := \gamma_{d} = \sum_{\alpha \in R_{+}} \textbf{\textit{k}}_{\scriptscriptstyle d}(\alpha).
\]  
Moreover, let $w_{\textbf{\textit{k}}_{\scriptscriptstyle d}}$ denote the weight function
\[ 
w_{\textbf{\textit{k}}_{\scriptscriptstyle d}}(x) = \prod_{\alpha \in R_{+}} \vert\langle \alpha, x \rangle \vert^{2\textbf{\textit{k}}_{\scriptscriptstyle d}(\alpha)}, \quad x \in \R^{d}.
\]
For every $p \in [1, +\infty]$, we denote by $L^{p}_{\textbf{\textit{k}}_{\scriptscriptstyle d}}(\R^{d})$ the space of measurable functions $f$ on $\R^{d}$ such that
\begin{eqnarray*}
\Vert f \Vert _{\textbf{\textit{k}}_{\scriptscriptstyle d}, p} &=& \left( \int_{\R^{d}} \vert f(x) \vert^{p} w_{\textbf{\textit{k}}_{\scriptscriptstyle d}}(x) \, dx \right)^{\frac{1}{p}} < + \infty, \quad \text{if } 1 \leq p < +\infty,\\
\Vert f \Vert _{\textbf{\textit{k}}_{\scriptscriptstyle d}, \infty} &=& \text{ess sup}_{x \in \R^{d}} \vert f(x) \vert < + \infty.
\end{eqnarray*}

The Dunkl transform on $L^{p}_{\textbf{\textit{k}}_{\scriptscriptstyle d}}(\R^{d})$ is defined by
\[
\mathcal{F}_{\textbf{\textit{k}}_{\scriptscriptstyle d}} (f) (y) := c_{\textbf{\textit{k}}_{\scriptscriptstyle d}} \int_{\mathbb{R}^{d}} f(x) E_{d, \textbf{\textit{k}}_{\scriptscriptstyle d}}(x, -i y) w_{\textbf{\textit{k}}_{\scriptscriptstyle d}}(x) \, dx,
\]
where the Mehta-type constant $c_{\textbf{\textit{k}}_{\scriptscriptstyle d}}$ is given by
\[ 
c_{\textbf{\textit{k}}_{\scriptscriptstyle d}} = \left( \int_{\R^{d}} e^{-\frac{\parallel x \parallel^{2}}{2}} w_{\textbf{\textit{k}}_{\scriptscriptstyle d}}(x) \, dx \right)^{-1}.
\]

\begin{remark}
In the case where $\textbf{\textit{k}}_{\scriptscriptstyle d}=0$, the weight function $w_{\textbf{\textit{k}}_{\scriptscriptstyle d}} = 1$, and the measure associated with the Dunkl operators coincides with the Lebesgue measure. Additionally, $E_{d, \textbf{\textit{k}}_{\scriptscriptstyle d}}(x, -iy) = e^{-i \langle x, y \rangle}$. Therefore, Dunkl analysis can be seen as a generalization of classical Fourier analysis.
\end{remark}
The Dunkl transform enjoys several fundamental properties, analogous to those of the Fourier transform, among which we mention the following:
\begin{enumerate}
\item[(i)] \textbf{Inversion theorem:} Let $f \in L^{p}_{\textbf{\textit{k}}_{\scriptscriptstyle d}}(\mathbb{R}^{d})$, such that $\mathcal{F}_{\textbf{\textit{k}}_{\scriptscriptstyle d}}(f) \in L^{p}_{\textbf{\textit{k}}_{\scriptscriptstyle d}}(\mathbb{R}^{d})$. Then,
\begin{align}\label{IVD}
f(x) = c_{\textbf{\textit{k}}_{\scriptscriptstyle d}} \int_{\mathbb{R}^{d}} \mathcal{F}_{\textbf{\textit{k}}_{\scriptscriptstyle d}}(f)(y) E_{d, \textbf{\textit{k}}_{\scriptscriptstyle d}}(x, i y) w_{\textbf{\textit{k}}_{\scriptscriptstyle d}}(y) \, dy.
\end{align}

\item[(ii)] \textbf{Plancherel formula:} For every $f \in L^{2}_{\textbf{\textit{k}}_{\scriptscriptstyle d}}(\mathbb{R}^{d})$, we have
\begin{align}\label{PF}
\int_{\mathbb{R}^{d}} \vert f(x) \vert^{2} w_{\textbf{\textit{k}}_{\scriptscriptstyle d}}(x) \, dx = \int_{\mathbb{R}^{d}} \vert \mathcal{F}_{\textbf{\textit{k}}_{\scriptscriptstyle d}}(f)(y) \vert^{2} w_{\textbf{\textit{k}}_{\scriptscriptstyle d}}(y) \, dy.
\end{align}

\item[(iii)] There exists a basis $\lbrace h_{v}: \, v \in \mathbb{Z}_{+}^{d} \rbrace$ of eigenfunctions of the Dunkl transform $\mathcal{F}_{\textbf{\textit{k}}_{\scriptscriptstyle d}}$ on $L^{2}_{\textbf{\textit{k}}_{\scriptscriptstyle d}}(\mathbb{R}^{d})$, satisfying
\begin{align}\label{tttt}
\mathcal{F}_{\textbf{\textit{k}}_{\scriptscriptstyle d}}(h_{v})(y) = 2^{\gamma + \frac{d}{2}} \, c^{-1}_{\textbf{\textit{k}}_{\scriptscriptstyle d}} \, (-i)^{\textbf{\textit{l}}( v )} \, h_{v}(y).
\end{align}

\item[(iv)] For every $\delta > 0$ and every $x \in \mathbb{R}^{d}$,
\begin{align}\label{ll66}
\mathcal{F}_{\textbf{\textit{k}}_{\scriptscriptstyle d}}(e^{-\delta \parallel x \parallel^{2}})(y) = \frac{1}{(2\delta)^{\gamma + \frac{d}{2}}} e^{-\frac{\parallel y \parallel^{2}}{4\delta}}.
\end{align}
\end{enumerate}

Let $x \in \mathbb{R}^{d}$. The Dunkl translation operator $f \mapsto \tilde{\tau}_{x}f$ is defined on $L^{2}_{\textbf{\textit{k}}_{\scriptscriptstyle d}}(\mathbb{R}^{d})$ by the relation
\[
\mathcal{F}_{\textbf{\textit{k}}_{\scriptscriptstyle d}} (\tilde{\tau}^{\textbf{\textit{k}}_{\scriptscriptstyle d}}_{x}f)(y) = E_{d, \textbf{\textit{k}}_{\scriptscriptstyle d}}(x, -i_{\mathbb{C}}y) \mathcal{F}_{\textbf{\textit{k}}_{\scriptscriptstyle d}} (f)(y) \quad \text{for all} \, y \in \mathbb{R}^{d}.
\]
Currently, explicit formulas for the generalized operator $\tilde{\tau}^{\textbf{\textit{k}}_{\scriptscriptstyle d}}$ are known in the following cases:
\begin{enumerate}
\item[\textbf{Case 1:}] When $d = 1$ and the reflection group $G$ is $\mathbb{Z}_{2}$.

Let $\Omega = \sqrt{y^{2} - 2xyt + x^{2}}$ and $\textbf{\textit{k}}$ be the unique multiplicity function associated with $R$. For continuous functions $f$ on $\mathbb{R}$, we have 
\begin{align}\label{EXPLICITD1}
\tilde{\tau}^{\textbf{\textit{k}}}_{x}f(y) = \frac{1}{2} \int_{-1}^{1} f(\Omega) \left( 1 + \frac{x - y}{\Omega} \right) \psi_{\textbf{\textit{k}}}(t) \, dt + \frac{1}{2} \int_{-1}^{1} f(-\Omega) \left( 1 - \frac{x - y}{\Omega} \right) \psi_{\textbf{\textit{k}}}(t) \, dt,
\end{align}
where
\[
\psi_{\textbf{\textit{k}}}(t) = \frac{\Gamma (\textbf{\textit{k}} + \frac{1}{2})}{\sqrt{\pi} \Gamma (\textbf{\textit{k}})} (1 + t)(1 - t^{2})^{\textbf{\textit{k}} - 1}.
\]
This formula was established by R\"{o}sler in \cite{rosler6}.
\item[\textbf{Case 2:}] When the function $f$ is radial and $f \in L^{1}_{\textbf{\textit{k}}_{\scriptscriptstyle d}}(\mathbb{R}^{d})$ such that $\mathcal{F}_{\textbf{\textit{k}}_{\scriptscriptstyle d}}(f) \in L^{1}_{\textbf{\textit{k}}_{\scriptscriptstyle d}}(\mathbb{R}^{d})$.

For all $x \in \mathbb{R}^{d}$, the Dunkl translation $\tilde{\tau}^{\textbf{\textit{k}}_{\scriptscriptstyle d}}_{x}$ is given by the following formula:
\begin{align}\label{EXPLICITD2}
\tilde{\tau}^{\textbf{\textit{k}}_{\scriptscriptstyle d}}_{x}f(y) = V_{\textbf{\textit{k}}_{\scriptscriptstyle d}} \left( F(\sqrt{\| x \|^{2} + 2 \langle x, \cdot \rangle + \| \cdot \|^{2}}) \right)(y),
\end{align}
where $F$ is a real-valued function on $\mathbb{R}^{d}$ such that $f(x) = F(\| x \|^{2})$, and $V_{\textbf{\textit{k}}_{\scriptscriptstyle d}}$ is the Dunkl intertwining operator. This formula was also derived by R\"{o}sler in \cite{rosler4}.
\end{enumerate}

\subsection{Fundamental concepts of Clifford algebras}
Clifford algebras, denoted $Cl_{p,q}$, form a family of associative algebras that generalize complex numbers and quaternions. They play a central role in geometry, analysis, and mathematical physics.

Let \(\mathbb{R}^{d}\) be a Euclidean space equipped with a non-degenerate quadratic form of signature \((p, q)\), where \(p + q = d\). Consider an orthogonal basis \(\{ e_{1}, \dots, e_{d} \}\) for \(\mathbb{R}^{p,q}\). The multiplication in the Clifford algebra \( Cl_{p,q} \), constructed over \(\mathbb{R}^{p,q}\), adheres to the following rules:
\[
e_i e_j + e_j e_i = 2 \eta_{ij}
\]
where \( \eta_{ij} \) is the metric of signature \((p,q)\), that is:
\[
\eta_{ij} = \begin{cases} 
+1 & \text{if } i = j \leq p \\
-1 & \text{if } i = j > p \\
0 & \text{if } i \neq j. 
\end{cases}
\]

The elements \(\{ e_{A} : A = \{j_{1}, \ldots, j_{k}\} \subset \{1, \ldots, d\} \}\), where \(e_{A} = e_{j_{1}} \dots e_{j_{k}}\) and \(e_{\emptyset} = 1\) (with \(e_{\emptyset}\) representing the empty set), form a basis for \( Cl_{p,q} \). Consequently, any element \( M \in Cl_{p,q} \) can be expressed as follows:
\begin{align}\label{CN}
M = \sum_{A} e_{A} M_{A}, \quad M_{A} \in \mathbb{R}.
\end{align}

Furthermore, the Clifford number \( M \) can be decomposed into its \(k\)-vector parts:
\begin{align*}
M = \sum_{k=0}^{d} \langle M \rangle_{k},
\end{align*}
where \(\langle M \rangle_{k} = \sum_{\mid A \mid = k} M_{A} e_{A}\) represents the \(k\)-vector part of \( M \). Each Clifford number comprises various \(k\)-vector components: the scalar part \(\langle M \rangle_{0} \in \mathbb{R}\), the vector part \(\langle M \rangle_{1}\), the bi-vector part \(\langle M \rangle_{2}\), and so forth, up to the pseudoscalar part \(\langle M \rangle_{d}\).

Important involutions in Clifford algebras include the reversion and the principal reverse. The reversion, denoted by \(\overline{M}\), changes the sign of every vector with a negative square in the basis decomposition of \(M\). This operation is defined as:
\[ \overline{e_{h_{1}}, \dots, e_{h_{k}}} = \varepsilon_{h_{1}} e_{h_{1}} \varepsilon_{h_{2}} e_{h_{2}} \dots \varepsilon_{h_{k}} e_{h_{k}} \]
where \(1 \leq h_{1} < \dots < h_{k} \leq d\), and \(\varepsilon_{k} = +1\) for \(k=1, \dots, p\), and \(\varepsilon_{k} = -1\) for \(k=p+1, \dots, d\).

The principal reverse, denoted by \(\widetilde{M}\), is given by:
\[ \widetilde{M} = \sum_{k=0}^{d} (-1)^{\frac{k(k-1)}{2}} \langle \overline{M} \rangle_{k}. \]

For \( M, N \in Cl_{p,q} \), the scalar product \( M * \widetilde{N} \) is defined by
\begin{align*}
 M * \widetilde{N} = \langle M \widetilde{N} \rangle_{0} = \sum_{A} M_{A} N_{A}.
\end{align*}
In particular, if \( M = N \), we obtain the modulus of a Clifford number \( M \in Cl_{p,q} \), defined as
\begin{align}\label{MCN}
\| M \|_{c}^{2} = M * \widetilde{M} = \sum_{A} M_{A}^{2}. 
\end{align}
\section{The Two-sided Clifford Dunkl Transform}
In this section, we present the two-sided Clifford Dunkl transform, an extension of the two-sided Clifford-Fourier transform. Our objective is to delve into its mathematical foundations and explore its inherent properties. 
%To provide a comprehensive understanding, we begin by defining the weighted function space \(L^{n}_{\textbf{\textit{k}}}(\mathbb{R}^{p,q};Cl_{p,q})\).\noindent In the following, we consider two root systems $R_{p}$ and $R_{q}$ such that $R= R_{p} \cup R_{q}$.
Before introducing the two-sided Clifford Dunkl transform, it is necessary to define some notations that will be employed throughout the remainder of this paper.

In what follows, we will consider: 
\begin{enumerate}
\item[$\bullet$] Two reduced root systems $R_{p}$ and $R_{q}$ such that $R =R_{p} \cup R_{q} $.
\item[$\bullet$] Two multiplicity functions $\textbf{\textit{k}}_{\scriptscriptstyle p}$ and $\textbf{\textit{k}}_{\scriptscriptstyle q}$ associated respectively with root systems $R_{p}$ and $R_{q}$.
\item[$\bullet$] Two Dunkl kernels $E_{p, \textbf{\textit{k}}_{\scriptscriptstyle p}}$ and $E_{q, \textbf{\textit{k}}_{\scriptscriptstyle q}}$ associated respectively with the root systems $R_{p}$ and $R_{q}$.
\item[$\bullet$] Two measures $d\mu^{\textbf{\textit{k}}_{p}}_{\scriptscriptstyle p}$ and $\mu^{\textbf{\textit{k}}_{q}}_{\scriptscriptstyle q}$ defined by $$d\mu^{\textbf{\textit{k}}_{p}}_{\scriptscriptstyle p}(\textbf{\textit{x}}_{1}) := w_{\textbf{\textit{k}}_{p}}(\textbf{\textit{x}}_{1})d^{p}\textbf{\textit{x}}_{1}  ~~ \text{ and} ~~ d\mu^{\textbf{\textit{k}}_{q}}_{q}(\textbf{\textit{x}}_{2}):=w_{ \textbf{\textit{k}}_{q}}(\textbf{\textit{x}}_{2})d^{q}\textbf{\textit{x}}_{2} ,$$ where $\textbf{\textit{x}}_{1} = (x_{1},\dots,x_{p})$ and $\textbf{\textit{x}}_{2} = (x_{p+1},\dots,x_{d}).$
\end{enumerate}
\subsection{Multivector-Valued Functions and the Space \( L^{n}_{\textbf{\textit{k}}_{\scriptscriptstyle  d}}(\mathbb{R}^{p,q};Cl_{p,q}) \)}
A multivector-valued function defined on \(\mathbb{R}^{p,q}\) takes values in the Clifford algebra \(Cl_{p,q}\). Any such function can be expressed as \( f(\textbf{\textit{x}}) = \sum_{A} f_{A}(\textbf{\textit{x}}) e_{A} \), where \( f_{A} \) are real-valued components. These functions can be decomposed as follows:
\begin{align}\label{cmf}
f(\textbf{\textit{x}}) = f_{0}(\textbf{\textit{x}}) + \sum_{k=1}^{d} f_{k}(\textbf{\textit{x}}) e_{k} + \sum_{1 \leq k \leq l \leq d} f_{kl}(\textbf{\textit{x}}) e_{k} e_{l} + \dots + f_{12 \dots d}(\textbf{\textit{x}}) e_{1} e_{2} \dots e_{d},
\end{align}
where the components \( f_{0}, f_{k}, f_{kl}, \dots, f_{12 \dots d} : \mathbb{R}^{d} \rightarrow \mathbb{R} \) are real-valued functions. 

The definition of the modulus can be extended to multivector-valued functions. The modulus of such a function is given by:
\begin{align*}
\| f(\textbf{\textit{x}}) \|_{c}^{2} = f_{0}^{2}(\textbf{\textit{x}}) + \sum_{k=1}^{d} f_{k}^{2}(\textbf{\textit{x}}) + \sum_{1 \leq k \leq l \leq d} f_{kl}^{2}(\textbf{\textit{x}}) + \dots + f_{12 \dots d}^{2}(\textbf{\textit{x}}).
\end{align*}

We define the inner product \((f, g)\) and the scalar product \(\langle f, g \rangle\) for multivector functions \( f, g : \mathbb{R}^{p,q} \rightarrow Cl_{p,q} \) as follows:
\begin{align}\label{mm}
(f, g) = \int_{\mathbb{R}^{p,q}} f(\textbf{\textit{x}}) \widetilde{g}(\textbf{\textit{x}}) \, d\mu_{p,q}(\textbf{\textit{x}}) = \sum_{A,B} e_{A} \widetilde{e_{B}} \int_{\mathbb{R}^{p,q}} f_{A}(\textbf{\textit{x}}) g_{B}(\textbf{\textit{x}}) \, d\mu_{p,q}(\textbf{\textit{x}}),
\end{align}
\begin{align}\label{mmm}
\langle f, g \rangle = \int_{\mathbb{R}^{p,q}} f(\textbf{\textit{x}}) * \widetilde{g}(\textbf{\textit{x}}) \, d\mu_{p,q}(\textbf{\textit{x}}) = \sum_{A} \int_{\mathbb{R}^{p,q}} f_{A}(\textbf{\textit{x}}) g_{A}(\textbf{\textit{x}}) \, d\mu_{p,q}(\textbf{\textit{x}}),
\end{align}
where $d\mu_{p,q}$ is the measure on $\mathbb{R}^{p,q}$ given by $d\mu_{p,q}(\textbf{\textit{x}}): =d\mu^{\textbf{\textit{k}}_{p}}_{p}(\textbf{\textit{x}}_{1})  d\mu^{\textbf{\textit{k}}_{q}}_{q}(\textbf{\textit{x}}_{2}) $.

Both (\ref{mm}) and (\ref{mmm}) lead to the \( L^{n}_{\textbf{\textit{k}}_{\scriptscriptstyle  d}}(\mathbb{R}^{p,q};Cl_{p,q}) \)-norm
\[
\vert f \vert_{\textbf{\textit{k}}_{\scriptscriptstyle  d},n} = \left( \int_{\mathbb{R}^{p,q}} \| f(\textbf{\textit{x}}) \|_{c}^{n} \, d\mu_{p,q}(\textbf{\textit{x}}) \right)^{\frac{1}{n}}, \quad \text{if } 1 \leq n < +\infty;
\]
and 
\[
\vert f \vert_{\textbf{\textit{k}}_{\scriptscriptstyle  d},\infty} = \text{ess} \, \sup_{\textbf{\textit{x}} \in \mathbb{R}^{p,q}} \| f(\textbf{\textit{x}}) \|_{c}.
\]
The linear spaces \( L^{n}_{\textbf{\textit{k}}_{\scriptscriptstyle  d}}(\mathbb{R}^{p,q};Cl_{p,q}) \) are defined as
\[
L^{n}_{\textbf{\textit{k}}_{\scriptscriptstyle  d}}(\mathbb{R}^{p,q};Cl_{p,q}) = \{ f : \mathbb{R}^{p,q} \rightarrow Cl_{p,q} : \vert f \vert_{\textbf{\textit{k}}_{\scriptscriptstyle  d},n} < \infty \}.
\]
\begin{remark}
It is important to note that the measure \( d\mu_{p,q} \) and the function space \( L^{n}_{\textbf{\textit{k}}_{\scriptscriptstyle  d}}(\mathbb{R}^{p,q};Cl_{p,q}) \) defined above differ from the measure \( w_{\textbf{\textit{k}}_{\scriptscriptstyle d}}(x)dx \) and the space \( L^{p}_{\textbf{\textit{k}}_{\scriptscriptstyle d}}(\R^{d}) \) introduced in Section 2. Specifically, \( d\mu_{p,q} \) is a product measure constructed from two distinct Dunkl measures associated with the root systems \( R_p \) and \( R_q \), and functions in \( L^{n}_{\textbf{\textit{k}}_{\scriptscriptstyle  d}}(\mathbb{R}^{p,q};Cl_{p,q}) \) take values in the Clifford algebra \( Cl_{p,q} \). In contrast, \( L^{p}_{\textbf{\textit{k}}_{\scriptscriptstyle d}}(\R^{d}) \) consists of scalar-valued functions on \( \mathbb{R}^d \) with the Dunkl measure \( w_{\textbf{\textit{k}}_{\scriptscriptstyle d}}(x)dx \). The two frameworks coincide only in the special case when \( p = d \) and \( q = 0 \) (or vice versa) and the Clifford algebra is reduced to the real numbers, but in general they are distinct.
\end{remark}
\subsection{Two-Sided Clifford Dunkl Transform}

\begin{definition}
Let \( a \) and \( b \) be elements in \( Cl_{p,q} \) such that \( a^{2} = b^{2} = -1 \). The two-sided Clifford Dunkl transform (CDT) of a function \( f \in L^{1}_{\textbf{\textit{k}}_{\scriptscriptstyle d}}(\mathbb{R}^{p,q};Cl_{p,q}) \), with respect to \(a\) and \(b\), is defined as:
\begin{equation}\label{GDQDT}
\mathcal{F}^{-a,-b} \lbrace f \rbrace (\textbf{\textit{y}}) = \int_{\mathbb{R}^{p,q}} E_{p, \textbf{\textit{k}}_{\scriptscriptstyle p}}(\textbf{\textit{x}}_{1},-a \textbf{\textit{y}}_{1})f(\textbf{\textit{x}}_{1},\textbf{\textit{x}}_{2})E_{q, \textbf{\textit{k}}_{\scriptscriptstyle q}}(\textbf{\textit{x}}_{2},-b \textbf{\textit{y}}_{2})\, d\mu_{p,q}(\textbf{\textit{x}}),
\end{equation}
where \( \textbf{\textit{y}}_{1} = (y_{1}, \dots, y_{p}) \) and \( \textbf{\textit{y}}_{2} = (y_{p+1}, \dots, y_{d}) \).
\end{definition}

The noncommutative nature of Clifford multiplication naturally leads to the development of two distinct forms of the CDT: the left-sided CDT and the right-sided CDT. The left-sided CDT is expressed as
\begin{equation}\label{LGDQDT}
\mathcal{F}_{l}^{-a,-b} \lbrace f \rbrace (\textbf{\textit{y}}) := \int_{\mathbb{R}^{p,q}} E_{p, \textbf{\textit{k}}_{\scriptscriptstyle  p}}(\textbf{\textit{x}}_{1},-a \textbf{\textit{y}}_{1})E_{q, \textbf{\textit{k}}_{\scriptscriptstyle  q}}(\textbf{\textit{x}}_{2},-b \textbf{\textit{y}}_{2})f(\textbf{\textit{x}}_{1},\textbf{\textit{x}}_{2})\, d\mu_{p,q}(\textbf{\textit{x}}),
\end{equation}
while the right-sided CDT is given by
\begin{equation}\label{RGDQDT}
\mathcal{F}_{r}^{-a,-b} \lbrace f \rbrace (\textbf{\textit{y}}) := \int_{\mathbb{R}^{p,q}} f(\textbf{\textit{x}}_{1},\textbf{\textit{x}}_{2})E_{p, \textbf{\textit{k}}_{\scriptscriptstyle  p}}(\textbf{\textit{x}}_{1},-a \textbf{\textit{y}}_{1})E_{q, \textbf{\textit{k}}_{\scriptscriptstyle  q}}(\textbf{\textit{x}}_{2},-b \textbf{\textit{y}}_{2})\, d\mu_{p,q}(\textbf{\textit{x}}).
\end{equation}
\noindent Note that by using the properties of the Dunkl transform and the Dunkl kernel, the study of the two previously introduced variants can be conducted similarly to the two-sided Clifford Dunkl transform. Consequently, the results obtained in the remainder of this paper can be readily extended to both variants.

\begin{remark}
The two-sided Clifford Dunkl transform exists for all integrable multivector-valued functions \( f \in L^{1}_{\textbf{\textit{k}}_{\scriptscriptstyle  d}}(\mathbb{R}^{p,q};Cl_{p,q}) \). Indeed, Equation (\ref{yy}) implies that
\begin{eqnarray*}
  \Vert \mathcal{F}^{-a,-b} \lbrace f \rbrace (\textbf{\textit{y}})   \Vert_{c}  
& \leq & \int_{\mathbb{R}^{p,q}}  \Vert E_{p, \textbf{\textit{k}}_{\scriptscriptstyle  p}}(\textbf{\textit{x}}_{1},-a \textbf{\textit{y}}_{1})  \Vert_{c}  \Vert f(\textbf{\textit{x}}_{1},\textbf{\textit{x}}_{2})  \Vert_{c}  \Vert E_{q, \textbf{\textit{k}}_{\scriptscriptstyle  q}}(\textbf{\textit{x}}_{2},-b \textbf{\textit{y}}_{2})  \Vert_{c} \, d\mu_{p,q}(\textbf{\textit{x}})\\
& \leq & \int_{\mathbb{R}^{p,q}}  \Vert f(\textbf{\textit{x}}_{1},\textbf{\textit{x}}_{2})  \Vert_{c}  \, d\mu_{p,q}(\textbf{\textit{x}}).
\end{eqnarray*}
Thus, \( \Vert \mathcal{F}^{-a,-b} \lbrace f \rbrace (\textbf{\textit{y}})  \Vert_{c} \) is finite, and therefore, the two-sided Clifford Dunkl transform is well-defined.
\end{remark}

\begin{remark}\label{hhjjkk}
The two-sided Clifford Dunkl transform with two square roots of \(-1\) in \( Cl_{p,q} \) encompasses several generalizations of the Fourier transform. Some notable examples include:
\begin{enumerate}
\item[$\bullet$]  If $a=b=i $, $\textbf{\textit{k}}_{p}=\textbf{\textit{k}}_{q}=0$  and $f \in L^{1}_{0}(\R^{d})$, the CDT reduces to the Euclidean Fourier transform on $\R^{d}$.
\item[$\bullet$] When $a=b=i$ and $f \in L^{1}_{\textbf{\textit{k}}_{\scriptscriptstyle d}}(\R^{d})$, the CDT is a Dunkl type transform.
\item[$\bullet$] In the case \( p=0 \) and \( q=2 \), the CDT reduces to the two-sided quaternion Dunkl transform \cite{key-fah}.
\item[$\bullet$] For \( p=0 \), \( q=2 \) and $\textbf{\textit{k}}_{p}=\textbf{\textit{k}}_{q}=0$, the CDT is just the quaternionic Fourier transform \cite{key-7, key-777}.
\item[$\bullet$] The CDT is the two-sided Clifford Fourier transform \cite{shi} when \( p=0 \) and $\textbf{\textit{k}}_{p}=\textbf{\textit{k}}_{q}=0$.
\item[$\bullet$] Depending on the choice of the phase functions \( u(\textbf{\textit{x}},\textbf{\textit{y}}) \) and \( v(\textbf{\textit{x}},\textbf{\textit{y}}) \), the general two-sided Clifford Fourier transform \cite[Definition (4.1)]{key-HITZ} is an example of CDT when \( u(\textbf{\textit{x}},\textbf{\textit{y}}) = \langle \textbf{\textit{x}}_{1},\textbf{\textit{y}}_{1} \rangle \) and \( v(\textbf{\textit{x}},\textbf{\textit{y}})= \langle \textbf{\textit{x}}_{2},\textbf{\textit{y}}_{2} \rangle \) and $\textbf{\textit{k}}_{p}=\textbf{\textit{k}}_{q}=0$.
\end{enumerate}
\end{remark}
\subsection{Essential Properties of the Two-Sided Clifford Dunkl Transform}
\begin{Th}[\textbf{Scalar Linearity}] 
Let \( \alpha, \beta \in \mathbb{R} \) and \( f, g \in L^{1}_{\textbf{\textit{k}}_{\scriptscriptstyle d}}(\mathbb{R}^{p,q};Cl_{p,q}) \). Then, the following holds:
\begin{align*}
\mathcal{F}^{-a,-b} \lbrace \alpha f + \beta g \rbrace (\textbf{y}) = \alpha \mathcal{F}^{-a,-b} \lbrace f \rbrace (\textbf{y}) + \beta \mathcal{F}^{-a,-b} \lbrace g \rbrace (\textbf{y}).
\end{align*}
\end{Th}

\begin{proof}
The result follows directly from the distributive property of the geometric product over addition, the commutativity of scalars, and the linearity of the integral.
\end{proof}
\begin{remark} 
For any multivector-valued function \(f\) in \(L^{1}_{\textbf{\textit{k}}_{\scriptscriptstyle d}}(\mathbb{R}^{p,q};Cl_{p,q})\), we have
\begin{align}\label{f11}
 f(\textbf{\textit{x}}_{1}, .) = c^{-1}_{\textbf{\textit{k}}_{\scriptscriptstyle p}}   \int_{\R^{p}} E_{p, \textbf{\textit{k}}_{\scriptscriptstyle p}}(\textbf{\textit{x}}_{1},i \textbf{\textit{y}}_{1})\mathcal{ F}_{\textbf{\textit{k}}_{\scriptscriptstyle p}}(f(\textbf{\textit{x}}_{1}, .))(\textbf{\textit{y}}_{1})d\mu^{\textbf{\textit{k}}_{p}}_{p}(\textbf{\textit{y}}_{1})    
\end{align}
and 
\begin{align}\label{f22}
 f(., \textbf{\textit{x}}_{2}) = c^{-1}_{\textbf{\textit{k}}_{\scriptscriptstyle q}} \int_{\R^{q}} \mathcal{ F}_{\textbf{\textit{k}}_{\scriptscriptstyle q}}(f(., \textbf{\textit{x}}_{2}))(\textbf{\textit{y}}_{2})E_{q, \textbf{\textit{k}}_{\scriptscriptstyle q}}(\textbf{\textit{x}}_{2},i \textbf{\textit{y}}_{2}) d\mu^{\textbf{\textit{k}}_{q}}_{q}(\textbf{\textit{y}}_{2}).
\end{align}
Indeed, any function \(f \in L^{1}_{\textbf{\textit{k}}_{\scriptscriptstyle d}}(\mathbb{R}^{p,q};Cl_{p,q})\) can be decomposed as
\begin{align*}
 f =  \sum_{A} f_{A} e_{A},
\end{align*}
where \(f_{A}\) are real-valued functions. Then, the Dunkl transform of the function \(f\) as a function on \(\R^{p}\) can be expressed as follows
\begin{eqnarray*}
\mathcal{ F}_{\textbf{\textit{k}}_{\scriptscriptstyle p}} (f(\textbf{\textit{x}}_{1},.)) (\textbf{\textit{y}}_{1})  &= & c_{\textbf{\textit{k}}_{\scriptscriptstyle p}} \int_{\mathbb{R}^{p}}f(\textbf{\textit{x}}_{1},.) E_{p, \textbf{\textit{k}}_{\scriptscriptstyle p}}(\textbf{\textit{x}}_{1},-i \textbf{\textit{y}}_{1})d\mu^{\textbf{\textit{k}}_{p}}_{p}(\textbf{\textit{x}}_{1})\\
&=& c_{\textbf{\textit{k}}_{\scriptscriptstyle p}} \sum_{A} e_{A} \int_{\mathbb{R}^{p}}f_{A}(\textbf{\textit{x}}_{1},.) E_{p, \textbf{\textit{k}}_{\scriptscriptstyle p}}(\textbf{\textit{x}}_{1},-i \textbf{\textit{y}}_{1})d\mu^{\textbf{\textit{k}}_{p}}_{p}(\textbf{\textit{x}}_{1}).
\end{eqnarray*}
Then, due to the inversion theorem for the Dunkl transform, we obtain
\begin{align*}
&\, \, \int_{\mathbb{R}^{p}} \mathcal{F}_{\textbf{\textit{k}}_{\scriptscriptstyle p}}(f(\textbf{\textit{x}}_{1},.))(\textbf{\textit{y}}_{1}) E_{p, \textbf{\textit{k}}_{\scriptscriptstyle p}}(\textbf{\textit{x}}_{1},i \textbf{\textit{y}}_{1}) d\mu^{\textbf{\textit{k}}_{p}}_{p}(\textbf{\textit{y}}_{1})\\
&= \int_{\mathbb{R}^{p}} \sum_{A} e_{A} \int_{\mathbb{R}^{p}} c_{\textbf{\textit{k}}_{\scriptscriptstyle p}} f_{A}(\textbf{\textit{x}}_{1},.) E_{p, \textbf{\textit{k}}_{\scriptscriptstyle p}}(\textbf{\textit{x}}_{1},-i \textbf{\textit{y}}_{1}) d\mu^{\textbf{\textit{k}}_{p}}_{p}(\textbf{\textit{x}}_{1}) E_{p, \textbf{\textit{k}}_{\scriptscriptstyle p}}(\textbf{\textit{x}}_{1},i \textbf{\textit{y}}_{1}) d\mu^{\textbf{\textit{k}}_{p}}_{p}(\textbf{\textit{y}}_{1}) \\
&= c^{-1}_{\textbf{\textit{k}}_{\scriptscriptstyle p}} \sum_{A} e_{A} \int_{\mathbb{R}^{p}} \int_{\mathbb{R}^{p}} c_{\textbf{\textit{k}}_{\scriptscriptstyle p}} f_{A}(\textbf{\textit{x}}_{1},.) E_{p, \textbf{\textit{k}}_{\scriptscriptstyle p}}(\textbf{\textit{x}}_{1},-i \textbf{\textit{y}}_{1}) E_{p, \textbf{\textit{k}}_{\scriptscriptstyle p}}(\textbf{\textit{x}}_{1},i \textbf{\textit{y}}_{1}) d\mu^{\textbf{\textit{k}}_{p}}_{p}(\textbf{\textit{x}}_{1}) d\mu^{\textbf{\textit{k}}_{p}}_{p}(\textbf{\textit{y}}_{1}) \\
&= c^{-1}_{\textbf{\textit{k}}_{\scriptscriptstyle p}} \sum_{A} e_{A} f_{A}(\textbf{\textit{x}}_{1},.) \\
&= c^{-1}_{\textbf{\textit{k}}_{\scriptscriptstyle p}} f(\textbf{\textit{x}}_{1},.).
\end{align*}
Similarly, we have
 
\begin{align*}
 f(., \textbf{\textit{x}}_{2}) =c^{-1}_{\textbf{\textit{k}}_{\scriptscriptstyle q}} \int_{\R^{q}} \mathcal{ F}_{\textbf{\textit{k}}_{\scriptscriptstyle q}}(f(., \textbf{\textit{x}}_{2}))(\textbf{\textit{y}}_{2})E_{q, \textbf{\textit{k}}_{\scriptscriptstyle q}}(\textbf{\textit{x}}_{2},i \textbf{\textit{y}}_{2}) d\mu^{\textbf{\textit{k}}_{q}}_{q}(\textbf{\textit{y}}_{2}).
\end{align*}
\end{remark}
\begin{proposition}[$\textbf{CDT  of a Gaussian function}$]
Let \( \delta \) be a positive scalar constant. Then,
\begin{align}\label{GGQDT}
\mathcal{F}^{-a,-b} \left\{ e^{- \delta \parallel \textbf{\textit{x}} \parallel_{c}^{2}} \right\} (\textbf{\textit{y}}) = \frac{1}{(2\delta)^{\gamma + \frac{d}{2}}} e^{-\frac{\parallel \textbf{\textit{y}} \parallel_{c}^{2}}{4\delta}}.
\end{align}
\end{proposition}

\begin{proof}
Using (\ref{ll66}) and the Fubini theorem, we obtain
\begin{align*}
\mathcal{F}^{-a,-b} \left\{ e^{- \delta \parallel \textbf{\textit{x}} \parallel_{c}^{2}} \right\} (\textbf{\textit{y}}) 
&= \int_{\mathbb{R}^{p,q}} E_{p, \textbf{\textit{k}}_{\scriptscriptstyle p}}(\textbf{\textit{x}}_{1}, -a \textbf{\textit{y}}_{1}) e^{- \delta \parallel \textbf{\textit{x}} \parallel_{c}^{2}} E_{q, \textbf{\textit{k}}_{\scriptscriptstyle q}}(\textbf{\textit{x}}_{2}, -b \textbf{\textit{y}}_{2}) d\mu_{p,q}(\textbf{\textit{x}}) \\
&= \int_{\mathbb{R}^{p}} E_{p, \textbf{\textit{k}}_{\scriptscriptstyle p}}(\textbf{\textit{x}}_{1}, -a \textbf{\textit{y}}_{1}) e^{- \delta \parallel \textbf{\textit{x}}_{1} \parallel_{c}^{2}} d\mu^{\textbf{\textit{k}}_{p}}_{p}(\textbf{\textit{x}}_{1}) \\
&\quad \times \int_{\mathbb{R}^{q}} e^{- \delta \parallel \textbf{\textit{x}}_{2} \parallel_{c}^{2}} E_{q, \textbf{\textit{k}}_{\scriptscriptstyle q}}(\textbf{\textit{x}}_{2}, -b \textbf{\textit{y}}_{2}) d\mu^{\textbf{\textit{k}}_{q}}_{q}(\textbf{\textit{x}}_{2}) \\
&= \frac{1}{(2\delta)^{\gamma_{p} + \frac{p}{2}}} e^{-\frac{\parallel \textbf{\textit{y}}_{1} \parallel_{c}^{2}}{4\delta}} 
\frac{1}{(2\delta)^{\gamma_{q} + \frac{q}{2}}} e^{-\frac{\parallel \textbf{\textit{y}}_{2} \parallel_{c}^{2}}{4\delta}} \\
&= \frac{1}{(2\delta)^{\gamma + \frac{d}{2}}} e^{-\frac{\parallel \textbf{\textit{y}} \parallel_{c}^{2}}{4\delta}}.
\end{align*}
\end{proof}
\subsubsection*{Inversion formula}
\begin{Th}[$\mathbf{Inversion ~ Formula}$]\label{IVF}
Let \( f \in L^{1}_{\textbf{\textit{k}}_{\scriptscriptstyle d}}(\mathbb{R}^{p,q};Cl_{p,q}) \) such that \( \mathcal{F}^{-a,-b} \lbrace f \rbrace \in L^{1}_{\textbf{\textit{k}}_{\scriptscriptstyle d}}(\mathbb{R}^{p,q};Cl_{p,q}) \). Then, the inversion formula is given by:
\begin{align}\label{inversin}
f(\textbf{\textit{x}}_{1},\textbf{\textit{x}}_{2}) = c^{2}_{\textbf{\textit{k}}_{\scriptscriptstyle p}}c^{2}_{\textbf{\textit{k}}_{\scriptscriptstyle q}} \int_{\mathbb{R}^{p,q}} E_{p, \textbf{\textit{k}}_{\scriptscriptstyle p}}(\textbf{\textit{x}}_{1},a \textbf{\textit{y}}_{1}) \mathcal{F}^{-a,-b} \lbrace f (\textbf{\textit{x}}_{1},\textbf{\textit{x}}_{2}) \rbrace(\textbf{\textit{y}}_{1},\textbf{\textit{y}}_{2}) E_{q, \textbf{\textit{k}}_{\scriptscriptstyle q}}(\textbf{\textit{x}}_{2},b \textbf{\textit{y}}_{2}) d\mu_{p,q}(\textbf{\textit{y}}).   
\end{align}
\end{Th}

\begin{proof}
From Equations (\ref{f11}) and (\ref{f22}), we obtain
\begin{align*}
& \, \, \int_{\mathbb{R}^{p,q}} E_{p, \textbf{\textit{k}}_{\scriptscriptstyle p}}(\textbf{\textit{x}}_{1},a \textbf{\textit{y}}_{1}) \mathcal{F}^{-a,-b} \lbrace f (\textbf{\textit{x}}_{1},\textbf{\textit{x}}_{2}) \rbrace (\textbf{\textit{y}}_{1},\textbf{\textit{y}}_{2}) E_{q, \textbf{\textit{k}}_{\scriptscriptstyle q}}(\textbf{\textit{x}}_{2},b \textbf{\textit{y}}_{2}) d\mu_{p,q}(\textbf{\textit{y}}) \\
&= \int_{\mathbb{R}^{p,q}} E_{p, \textbf{\textit{k}}_{\scriptscriptstyle p}}(\textbf{\textit{x}}_{1},a \textbf{\textit{y}}_{1}) \int_{\mathbb{R}^{p,q}} E_{p, \textbf{\textit{k}}_{\scriptscriptstyle p}}(\textbf{\textit{x}}_{1},-a \textbf{\textit{y}}_{1}) f(\textbf{\textit{x}}_{1},\textbf{\textit{x}}_{2}) E_{q, \textbf{\textit{k}}_{\scriptscriptstyle q}}(\textbf{\textit{x}}_{2},-b \textbf{\textit{y}}_{2}) d\mu_{p,q}(\textbf{\textit{x}}) \\
& \quad \times E_{q, \textbf{\textit{k}}_{\scriptscriptstyle q}}(\textbf{\textit{x}}_{2},b \textbf{\textit{y}}_{2}) d\mu_{p,q}(\textbf{\textit{y}}) \\
&= c^{-1}_{\textbf{\textit{k}}_{\scriptscriptstyle p}} c_{\textbf{\textit{k}}_{\scriptscriptstyle p}} \int_{\mathbb{R}^{p}} \int_{\mathbb{R}^{p}} E_{p, \textbf{\textit{k}}_{\scriptscriptstyle p}}(\textbf{\textit{x}}_{1},a \textbf{\textit{y}}_{1}) E_{p, \textbf{\textit{k}}_{\scriptscriptstyle p}}(\textbf{\textit{x}}_{1},-a \textbf{\textit{y}}_{1}) d\mu^{\textbf{\textit{k}}_{p}}_{p}(\textbf{\textit{x}}_{1}) d\mu^{\textbf{\textit{k}}_{p}}_{p}(\textbf{\textit{y}}_{1}) \\
&\quad \times c^{-1}_{\textbf{\textit{k}}_{\scriptscriptstyle q}} c_{\textbf{\textit{k}}_{\scriptscriptstyle q}} \int_{\mathbb{R}^{q}} \int_{\mathbb{R}^{q}} f(\textbf{\textit{x}}_{1},\textbf{\textit{x}}_{2}) E_{q, \textbf{\textit{k}}_{\scriptscriptstyle q}}(\textbf{\textit{x}}_{2},-b \textbf{\textit{y}}_{2}) E_{q, \textbf{\textit{k}}_{\scriptscriptstyle q}}(\textbf{\textit{x}}_{2},b \textbf{\textit{y}}_{2}) d\mu^{\textbf{\textit{k}}_{q}}_{q}(\textbf{\textit{x}}_{2}) d\mu^{\textbf{\textit{k}}_{q}}_{q}(\textbf{\textit{y}}_{2}) \\
&= c^{-1}_{\textbf{\textit{k}}_{\scriptscriptstyle p}} c^{-1}_{\textbf{\textit{k}}_{\scriptscriptstyle p}} c^{-1}_{\textbf{\textit{k}}_{\scriptscriptstyle q}} \int_{\mathbb{R}^{p}} \int_{\mathbb{R}^{p}} E_{p, \textbf{\textit{k}}_{\scriptscriptstyle p}}(\textbf{\textit{x}}_{1},a \textbf{\textit{y}}_{1}) E_{p, \textbf{\textit{k}}_{\scriptscriptstyle p}}(\textbf{\textit{x}}_{1},-a \textbf{\textit{y}}_{1}) f(\textbf{\textit{x}}_{1},\textbf{\textit{x}}_{2}) d\mu^{\textbf{\textit{k}}_{p}}_{p}(\textbf{\textit{x}}_{1}) d\mu^{\textbf{\textit{k}}_{p}}_{p}(\textbf{\textit{y}}_{1}) \\
&= c^{-1}_{\textbf{\textit{k}}_{\scriptscriptstyle p}} c^{-1}_{\textbf{\textit{k}}_{\scriptscriptstyle p}} c^{-1}_{\textbf{\textit{k}}_{\scriptscriptstyle q}} c^{-1}_{\textbf{\textit{k}}_{\scriptscriptstyle q}} f(\textbf{\textit{x}}_{1},\textbf{\textit{x}}_{2}).
\end{align*}
The proof is complete.

\end{proof}
\subsubsection*{Plancherel Formula}
Before delving into the Plancherel formula, we first establish some fundamental results concerning the Clifford eigenfunctions of certain function families. These results are instrumental in proving the Plancherel formula for the CDT.
\begin{lemma}\label{CGFL}
The set
\begin{align}\label{CGF}
\lbrace h_{u}(\textbf{\textit{x}}_{1}) h_{v}(\textbf{\textit{x}}_{2}) : u \in \mathbb{Z}_{+}^{p}, \, v \in \mathbb{Z}_{+}^{q} \rbrace
\end{align}
is a generating family for the weighted space \( L^{2}_{\textbf{\textit{k}}_{\scriptscriptstyle d}}(\mathbb{R}^{p,q};Cl_{p,q}) \).
\end{lemma}
\begin{proof}
Let \( f \in L^{2}_{\textbf{\textit{k}}_{\scriptscriptstyle d}}(\mathbb{R}^{p,q};Cl_{p,q}) \) be a multivector-valued function. We can express \( f \) as a sum of real-valued component functions \( f_{A} \) multiplied by the corresponding basis elements \( e_{A} \):
\begin{align}\label{MVF}
f(\textbf{\textit{x}}_{1}, \textbf{\textit{x}}_{2}) = \sum_{A} f_{A}(\textbf{\textit{x}}_{1}, \textbf{\textit{x}}_{2}) e_{A}.
\end{align}
Here, \( f_{A} \) are real-valued functions, and it follows that \( f_{A} \in L^{2}_{\textbf{\textit{k}}_{\scriptscriptstyle d}}(\mathbb{R}^{d}) \).

For any fixed \( \textbf{\textit{x}}_{2} \in \mathbb{R}^{q} \), each square-integrable function \( f_{A}(\textbf{\textit{x}}_{1}, \textbf{\textit{x}}_{2}) \) can be expanded using the eigenfunctions  of the Dunkl transform \( \lbrace h_{v}: \, v \in \mathbb{Z}_{+}^{p} \rbrace \) as follows:
\begin{align}\label{RVF}
f_{A}(\textbf{\textit{x}}_{1}, \textbf{\textit{x}}_{2}) = \sum_{v \in \mathbb{Z}_{+}^{p}} M_{A}^{v}(\textbf{\textit{x}}_{2}) h_{v}(\textbf{\textit{x}}_{1}),
\end{align}
where the coefficients \( M_{A}^{v}(\textbf{\textit{x}}_{2}) \) are given by the integral
\begin{align*}
M_{A}^{v}(\textbf{\textit{x}}_{2}) = \int_{\mathbb{R}^{p}} h_{v}(\textbf{\textit{x}}_{1}) f_{A}(\textbf{\textit{x}}_{1}, \textbf{\textit{x}}_{2}) d\mu^{\textbf{\textit{k}}_{p}}_{p}(\textbf{\textit{x}}_{1}).
\end{align*}
Since \( M_{A}^{v}(\textbf{\textit{x}}_{2}) \) is square-integrable as a function of \( \textbf{\textit{x}}_{2} \), we can further expand it as
\begin{align*}
M_{A}^{v}(\textbf{\textit{x}}_{2}) = \sum_{u \in \mathbb{Z}_{+}^{q}} M_{A}^{v,u} h_{u}(\textbf{\textit{x}}_{2}).
\end{align*}
Substituting this expression back into Equation~\eqref{RVF} yields
\begin{align}\label{GCF}
f_{A}(\textbf{\textit{x}}_{1}, \textbf{\textit{x}}_{2}) = \sum_{v \in \mathbb{Z}_{+}^{p}} \sum_{u \in \mathbb{Z}_{+}^{q}} M_{A}^{v,u} h_{u}(\textbf{\textit{x}}_{2}) h_{v}(\textbf{\textit{x}}_{1}).
\end{align}
By substituting the expression for \( f_{A} \) from Equation~\eqref{GCF} into Equation~\eqref{MVF}, we derive
\begin{align}
f(\textbf{\textit{x}}_{1}, \textbf{\textit{x}}_{2}) &= \sum_{A} \sum_{v \in \mathbb{Z}_{+}^{p}} \sum_{u \in \mathbb{Z}_{+}^{q}} M_{A}^{v,u} h_{u}(\textbf{\textit{x}}_{2}) h_{v}(\textbf{\textit{x}}_{1}) e_{A}\\ \nonumber
&= \sum_{v \in \mathbb{Z}_{+}^{p}} \sum_{u \in \mathbb{Z}_{+}^{q}} h_{u}(\textbf{\textit{x}}_{2}) h_{v}(\textbf{\textit{x}}_{1}) \sum_{A} M_{A}^{v,u} e_{A},
\end{align}
which completes the proof.
\end{proof}
\begin{lemma}\label{EBL}
The eigenfunctions of the CDT are given by the functions $h_{u} h_{v}$ with the following (Clifford) eigenvalues: 
\begin{align}\label{EB}
\mathcal{ F}^{-a,-b} \lbrace h_{u} h_{v}  \rbrace ( \textbf{\textit{y}}_{1},  \textbf{\textit{y}}_{2}) = c^{-2}_{\textbf{\textit{k}}_{\scriptscriptstyle p}} \, c^{-2}_{\textbf{\textit{k}}_{\scriptscriptstyle q}} \, 2^{ \gamma_{p} +\frac{p}{2}} \,  2^{ \gamma_{q} +\frac{q}{2}}  (-a)^{\textbf{\textit{l}}( v ) }  (-b)^{\textbf{\textit{l}}( u ) } h_{v}(\textbf{\textit{y}}_{1})h_{u}(\textbf{\textit{y}}_{2}).
\end{align}
\end{lemma}
\begin{proof}
Applying Fubini's theorem and Formula (\ref{tttt}), we obtain
\begin{eqnarray*}
\mathcal{ F}^{-a,-b} \lbrace h_{u} h_{v}  \rbrace (\textbf{\textit{y}}_{1},  \textbf{\textit{y}}_{2}) &=& \int_{\mathbb{R}^{p,q}} E_{p, \textbf{\textit{k}}_{\scriptscriptstyle p}}(\textbf{\textit{x}}_{1},-a \textbf{\textit{y}}_{1}) h_{u}(\textbf{\textit{x}}_{2}) h_{v}(\textbf{\textit{x}}_{1})  E_{q, \textbf{\textit{k}}_{\scriptscriptstyle q}}(\textbf{\textit{x}}_{2},-b \textbf{\textit{y}}_{2})d\mu_{p, q}(\textbf{\textit{x}})\\
 &=&  c^{-1}_{\textbf{\textit{k}}_{\scriptscriptstyle p}} \, c_{\textbf{\textit{k}}_{\scriptscriptstyle p}} \int_{\mathbb{R}^{p}} E_{p, \textbf{\textit{k}}_{\scriptscriptstyle p}}(\textbf{\textit{x}}_{1},-a \textbf{\textit{y}}_{1})h_{v}(\textbf{\textit{x}}_{1}) d\mu^{\textbf{\textit{k}}_{p}}_{p}(\textbf{\textit{x}}_{1}) \\ 
 & \times &  c^{-1}_{\textbf{\textit{k}}_{\scriptscriptstyle q}} \, c_{\textbf{\textit{k}}_{\scriptscriptstyle q}} 
  \int_{\mathbb{R}^{q}} E_{q, \textbf{\textit{k}}_{\scriptscriptstyle q}}(\textbf{\textit{x}}_{2},-b \textbf{\textit{y}}_{2}) h_{u}(\textbf{\textit{x}}_{2})  d\mu^{\textbf{\textit{k}}_{q}}_{q}(\textbf{\textit{x}}_{2})\\
 &=& c^{-1}_{\textbf{\textit{k}}_{\scriptscriptstyle p}} \, 2^{ \gamma_{p} +\frac{p}{2}} \, c^{-1}_{\textbf{\textit{k}}_{\scriptscriptstyle p}}   (-a)^{\textbf{\textit{l}}( v ) }  h_{v}(\textbf{\textit{y}}_{1})  c^{-1}_{\textbf{\textit{k}}_{\scriptscriptstyle q}} \, 2^{ \gamma_{q} +\frac{q}{2}} \, c^{-1}_{\textbf{\textit{k}}_{\scriptscriptstyle q}}   (-b)^{\textbf{\textit{l}}( u )} h_{u}(\textbf{\textit{y}}_{2})\\
 &=& c^{-2}_{\textbf{\textit{k}}_{\scriptscriptstyle p}} \, c^{-2}_{\textbf{\textit{k}}_{\scriptscriptstyle q}} \, 2^{ \gamma_{p} +\frac{p}{2}}  \, 2^{ \gamma_{q} +\frac{q}{2}} (-a)^{\textbf{\textit{l}}( v ) }  (-b)^{\textbf{\textit{l}} (u ) } h_{v}(\textbf{\textit{y}}_{1})h_{u}(\textbf{\textit{y}}_{2}).
\end{eqnarray*}
\end{proof}
\begin{Th}[ $ \mathbf{ Plancherel ~~ formula ~~ for ~~ CDT} $]\label{CPT}
For every $f \in L^{2}_{\textbf{\textit{k}}_{d}}(\mathbb{R}^{p,q};Cl_{p,q})$, we have
\begin{align}
\vert \mathcal{ F}^{-a,-b} \lbrace f \rbrace \vert^{2}_{\textbf{\textit{k}}_{\scriptscriptstyle d},2} = ( c^{-2}_{\textbf{\textit{k}}_{\scriptscriptstyle p}} \, c^{-2}_{\textbf{\textit{k}}_{\scriptscriptstyle q}} \, 2^{ \gamma_{p} +\frac{p}{2}} \,  2^{ \gamma_{q} +\frac{q}{2}} )^{2} \vert f \vert^{2}_{\textbf{\textit{k}}_{\scriptscriptstyle d},2}.  
\end{align}
\end{Th}

\begin{proof}
First, we calculate the action of the CDT transform on a multivector-valued function $f \in L^{2}_{\textbf{\textit{k}}_{\scriptscriptstyle d}}(\mathbb{R}^{p,q};Cl_{p,q})$.\\
By Lemmas \ref{CGFL} and \ref{EBL}, we have:
\begin{eqnarray*}
\mathcal{ F}^{-a,-b} \lbrace f  \rbrace (\textbf{\textit{y}}) &=& \sum_{v \in \mathbb{Z}_{+}^{p}} \sum_{u \in \mathbb{Z}_{+}^{q}} \mathcal{ F}^{-a,-b} \lbrace h_{u}(\textbf{\textit{x}}_{2}) h_{v}(\textbf{\textit{x}}_{1}) M_{v, u}  \rbrace (\textbf{\textit{y}}) \\
&=& \sum_{v \in \mathbb{Z}_{+}^{p}} \sum_{u \in \mathbb{Z}_{+}^{q}} c^{-2}_{\textbf{\textit{k}}_{\scriptscriptstyle p}} \, c^{-2}_{\textbf{\textit{k}}_{\scriptscriptstyle q}} \, 2^{ \gamma_{p} +\frac{p}{2}} \,  2^{ \gamma_{q} +\frac{q}{2}}  (-a)^{\textbf{\textit{l}} (v) } M_{v, u} (-b)^{\textbf{\textit{l}} (u ) } h_{v}(\textbf{\textit{y}}_{1})h_{u}(\textbf{\textit{y}}_{2}),
\end{eqnarray*}
where the Clifford numbers $M_{v, u}$ do not necessarily commute with $a$ and $b$.\\
Thus, we calculate:
\begin{align*}
&\, \,  \langle \mathcal{F}^{-a,-b} \lbrace f \rbrace (\textbf{\textit{y}}) ; \mathcal{F}^{-a,-b} \lbrace f \rbrace (\textbf{\textit{y}}) \rangle \\
&= \int_{\mathbb{R}^{p,q}} \mathcal{F}^{-a,-b} \lbrace f \rbrace (\textbf{\textit{y}}) * \widetilde{\mathcal{F}^{-a,-b} \lbrace f \rbrace (\textbf{\textit{y}})} \, d\mu_{p,q}(\textbf{\textit{y}}) \\
&= \left( c^{-2}_{\textbf{\textit{k}}_{\scriptscriptstyle p}} \, c^{-2}_{\textbf{\textit{k}}_{\scriptscriptstyle q}} \, 2^{\gamma_{p} + \frac{p}{2}} \, 2^{\gamma_{q} + \frac{q}{2}} \right)^{2} \sum_{v \in \mathbb{Z}_{+}^{p}} \sum_{u \in \mathbb{Z}_{+}^{q}} (-a)^{\textbf{\textit{l}}(v)} M_{v,u} (-b)^{\textbf{\textit{l}}(u)} \widetilde{(-a)^{\textbf{\textit{l}}(v)} M_{v,u} (-b)^{\textbf{\textit{l}}(u)}} \\
&\quad \times \int_{\mathbb{R}^{p,q}} h_{v}(\textbf{\textit{y}}_{1}) h_{u}(\textbf{\textit{y}}_{2}) * \widetilde{h_{v}(\textbf{\textit{y}}_{1}) h_{u}(\textbf{\textit{y}}_{2})} \, d\mu^{\textbf{\textit{k}}_{p}}_{p}(\textbf{\textit{y}}_{1}) \, d\mu^{\textbf{\textit{k}}_{q}}_{q}(\textbf{\textit{y}}_{2}) \\
&= \left( c^{-2}_{\textbf{\textit{k}}_{\scriptscriptstyle p}} \, c^{-2}_{\textbf{\textit{k}}_{\scriptscriptstyle q}} \, 2^{\gamma_{p} + \frac{p}{2}} \, 2^{\gamma_{q} + \frac{q}{2}} \right)^{2} \sum_{v \in \mathbb{Z}_{+}^{p}} \sum_{u \in \mathbb{Z}_{+}^{q}} (-a)^{\textbf{\textit{l}}(v)} M_{v,u} (-b)^{\textbf{\textit{l}}(u)} \widetilde{(-b)^{\textbf{\textit{l}}(u)}} \widetilde{M_{v,u}} \widetilde{(-a)^{\textbf{\textit{l}}(v)}} \\
&\quad \times \int_{\mathbb{R}^{p}} h_{v}(\textbf{\textit{y}}_{1}) \widetilde{h_{v}(\textbf{\textit{y}}_{1})} \, d\mu^{\textbf{\textit{k}}_{p}}_{p}(\textbf{\textit{y}}_{1}) \times \int_{\mathbb{R}^{q}} h_{u}(\textbf{\textit{y}}_{2}) \widetilde{h_{u}(\textbf{\textit{y}}_{2})} \, d\mu^{\textbf{\textit{k}}_{q}}_{q}(\textbf{\textit{y}}_{2}) \\
&= \left( c^{-2}_{\textbf{\textit{k}}_{\scriptscriptstyle p}} \, c^{-2}_{\textbf{\textit{k}}_{\scriptscriptstyle q}} \, 2^{\gamma_{p} + \frac{p}{2}} \, 2^{\gamma_{q} + \frac{q}{2}} \right)^{2} \sum_{v \in \mathbb{Z}_{+}^{p}} \sum_{u \in \mathbb{Z}_{+}^{q}} \Vert M_{v,u} \Vert_{c}^{2}.
\end{align*}
Similarly, we calculate:

\begin{eqnarray*}
\langle f; f \rangle & = & \sum_{v \in \mathbb{Z}_{+}^{p}} \sum_{u \in \mathbb{Z}_{+}^{q}} M_{v, u} \widetilde{M_{v, u} }  \int_{\mathbb{R}^{p}} h_{v}(\textbf{\textit{x}}_{1})\widetilde{h_{v}(\textbf{\textit{x}}_{1})}d\mu^{\textbf{\textit{k}}_{p}}_{p}(\textbf{\textit{x}}_{1})\\
& \times & \int_{\mathbb{R}^{q}} h_{u}(\textbf{\textit{x}}_{2})\widetilde{h_{u}(\textbf{\textit{x}}_{2})}d\mu^{\textbf{\textit{k}}_{q}}_{q}(\textbf{\textit{x}}_{2})\\
&=&  \sum_{v \in \mathbb{Z}_{+}^{p}} \sum_{u \in \mathbb{Z}_{+}^{q}} \Vert M_{v, u} \Vert_{c}^{2}.
\end{eqnarray*}
Thus, the proof is complete.
\end{proof}
\begin{remark}
 In signal processing, the previous theorem states that the signal energy is preserved by the CDT.
 \end{remark}
\subsubsection*{Translation and convolution for the CDT}
One of the most fundamental concepts in Fourier theory  is the convolution, defined as
\begin{align*}
(f*g)(x) =  \int_{\mathbb{R}^{d}} f(y)g(x-y)dy,
\end{align*}
 which plays a crucial role in signal processing, such as edge detection, sharpening, and smoothing in image processing.\\
We introduce the Clifford convolution of the CDT, which extends the classical convolution to Clifford algebra. Let us first define the convolution of two multivector valued functions. To achieve this, we introduce the notion of generalized translation.  
 \begin{definition}
Let $ f \in  L^{2}_{\textbf{\textit{k}}_{\scriptscriptstyle d}}(\mathbb{R}^{p,q};Cl_{p,q}) $ and $ \textbf{\textit{z}} \in \mathbb{R}^{p,q} $ be given. The generalized translation operator $ \tau_{\textbf{\textit{z}}} f $  is defined by:
\begin{align}\label{DC}
\mathcal{ F}^{-a,-b} \lbrace  \tau_{\textbf{\textit{z}}} f  \rbrace (\textbf{\textit{y}}) = E_{p, \textbf{\textit{k}}_{\scriptscriptstyle p}}(\textbf{\textit{z}}_{1},-a \textbf{\textit{y}}_{1})  \mathcal{ F}^{-a,-b} \lbrace f  \rbrace (\textbf{\textit{y}}) E_{q, \textbf{\textit{k}}_{\scriptscriptstyle q}}(\textbf{\textit{z}}_{2},-b \textbf{\textit{y}}_{2}).
\end{align}
\end{definition}
\begin{remark}
If $\textbf{\textit{k}}_{p}=\textbf{\textit{k}}_{q}=0$ we get a translation operator for the two-sided Clifford Fourier transform, and if in addition $p=0$ and $q=2$ the previous definition coincides with that of the translation operator associated with the quaternionic Fourier transform(see \cite{DBY1}).
\end{remark}
\begin{proposition}
The generalized translation $\tau_\textbf{\textit{z}}:  L^{2}_{\textbf{\textit{k}}_{\scriptscriptstyle d}}(\mathbb{R}^{p,q};Cl_{p,q}) \rightarrow  L^{2}_{\textbf{\textit{k}}_{\scriptscriptstyle d}}(\mathbb{R}^{p,q};Cl_{p,q})$ is a bounded operator, and we have
\begin{align}
\vert \tau_\textbf{\textit{z}}f \vert^{2}_{\textbf{\textit{k}}_{\scriptscriptstyle d},2} \leq \vert f \vert^{2}_{\textbf{\textit{k}}_{\scriptscriptstyle d},2}, ~~ f \in  L^{2}_{\textbf{\textit{k}}_{\scriptscriptstyle d}}(\mathbb{R}^{p,q};Cl_{p,q}).
\end{align}
\end{proposition}

 \begin{proof}
A direct application of Plancherel Theorem \ref{CPT} and formula (\ref{yy}) gives the result.
\end{proof}
 \begin{remark}
According to the inverse CDT, we obtain the following explicit formula for the generalized translation operator $ \tau$:
\begin{eqnarray}\label{EXPLICITD}
\tau_{\textbf{\textit{z}}} f(\textbf{\textit{x}}) \nonumber
& = & c^{2}_{\textbf{\textit{k}}_{p}} c^{2}_{\textbf{\textit{k}}_{q}} \mathcal{ F}^{a,b} \lbrace E_{p, \textbf{\textit{k}}_{p}}(\textbf{\textit{z}}_{1},-a \textbf{\textit{y}}_{1})  \mathcal{ F}^{-a,-b} \lbrace f \rbrace (\textbf{\textit{y}}) E_{q, \textbf{\textit{k}}_{q}}(\textbf{\textit{z}}_{2},-b \textbf{\textit{y}}_{2}) \rbrace(\textbf{\textit{x}}) \\ \nonumber
& = & c^{2}_{\textbf{\textit{k}}_{p}} c^{2}_{\textbf{\textit{k}}_{q}} \int_{\mathbb{R}^{p,q}}   E_{p, \textbf{\textit{k}}_{p}}(\textbf{\textit{z}}_{1},-a \textbf{\textit{y}}_{1}) E_{p, \textbf{\textit{k}}_{p}}(\textbf{\textit{x}}_{1},a \textbf{\textit{y}}_{1}) \mathcal{ F}^{-a,-b} \lbrace f  \rbrace (\textbf{\textit{y}}) \\ 
& & \times E_{q, \textbf{\textit{k}}_{q}}(\textbf{\textit{x}}_{2},b \textbf{\textit{y}}_{2}) E_{q, \textbf{\textit{k}}_{q}}(\textbf{\textit{z}}_{2},-b \textbf{\textit{y}}_{2}) d\mu_{p,q}(\textbf{\textit{y}}).
\end{eqnarray}
\end{remark}

Now, let us define the function space necessary for the next proposition. Let $\mathcal{B}^{p,q}_{\textbf{\textit{k}}_{\scriptscriptstyle d}}$ denote the subspace of functions $f$ in $L^{1}_{\textbf{\textit{k}}_{\scriptscriptstyle d}}(\mathbb{R}^{p,q};Cl_{p,q})$, such that $ \mathcal{ F}^{-a,-b} \lbrace f \rbrace (\textbf{\textit{y}}) \in  L^{1}_{\textbf{\textit{k}}_{\scriptscriptstyle d}}(\mathbb{R}^{p,q};Cl_{p,q})$.

The following proposition gives the relation between the generalized translation operator $\tau$ and the Dunkl translation $\tilde{\tau}^{\textbf{\textit{k}}_{\scriptscriptstyle d}}$.

\begin{proposition}
Let $f$ be in $\mathcal{B}^{p,q}_{\textbf{\textit{k}}_{\scriptscriptstyle d}}$. Then, the generalized translation operator $\tau$ satisfies the following relation:
\begin{align}
\tau_{\textbf{\textit{z}}} f(\textbf{\textit{x}}_{1}, \textbf{\textit{x}}_{2}) = c^{2}_{\textbf{\textit{k}}_{q}} c^{2}_{\textbf{\textit{k}}_{p}} \tilde{\tau}^{\textbf{\textit{k}}_{p}}_{\textbf{\textit{z}}_{1}} \left( \tilde{\tau}^{\textbf{\textit{k}}_{q}}_{\textbf{\textit{z}}_{2}} f( \cdot , \textbf{\textit{x}}_{2}) \right)(\textbf{\textit{x}}_{1}).
\end{align}
\end{proposition}

\begin{proof}
By substituting the expression of $\mathcal{ F}^{-a,-b} \lbrace f \rbrace (\textbf{\textit{y}})$ into Equation \ref{EXPLICITD}, and subsequently applying Fubini's theorem to interchange the order of integration, we obtain the desired result.
\end{proof}

After defining the generalized translation operator analogously to the two-sided QDT, we derive an explicit formula for it. However, it is essential first to establish the definition of a radial function, which takes values in a Clifford algebra. Following this, we will investigate the properties of the generalized translation operator $\tau_{\textbf{\textit{z}}}$ on radial functions, as they are significant in various applications due to their inherent symmetries. We provide the definition of a radial function and a corresponding proposition that describes the action of the generalized translation operator on such functions. These results will deepen our understanding of $\tau_{\textbf{\textit{z}}}$ and its interactions with radial functions.

\begin{definition}
A function $f : \mathbb{R}^{p,q} \longrightarrow Cl_{p,q}$ is called radial if each component $f_{A}$ is radial.
\end{definition}
\begin{proposition}
Let $f \in B^{p,q}_{\textbf{\textit{k}}_{\scriptscriptstyle d}}$ be a radial function, which can be decomposed with respect to the variables $\textbf{\textit{x}}_1$ and $\textbf{\textit{x}}_2$ as:
$$f( \textbf{\textit{x}}) = f(\textbf{\textit{x}}_{1}, \textbf{\textit{x}}_{2} ) = F(\Vert \textbf{\textit{x}} \Vert_{c}) = F_{1}(\Vert \textbf{\textit{x}}_{1} \Vert_{c}) F_{2}(\Vert \textbf{\textit{x}}_{2} \Vert_{c}),$$
where 
$$F_{1}(\Vert \textbf{\textit{x}}_{1} \Vert_{c}) = \sum_{A} e_{A} F^{A}_{1}(\Vert \textbf{\textit{x}}_{1} \Vert_{c}), \quad F_{2}(\Vert \textbf{\textit{x}}_{2} \Vert_{c}) = \sum_{B} e_{B} F^{B}_{2}(\Vert \textbf{\textit{x}}_{2} \Vert_{c}).$$
Then, the generalized translation of $f$ is given by:
\begin{eqnarray*}
\tau_{\textbf{\textit{z}}} f(\textbf{\textit{x}}) & = & \sum_{A}  V_{\textbf{\textit{k}}_{\scriptscriptstyle p}} \left(  F^{A}_{1}(\sqrt{\Vert \textbf{\textit{z}}_{1} \Vert_{c}^{2} + 2 \langle \textbf{\textit{z}}_{1}, . \rangle + \Vert . \Vert_{c}^{2}})\right) (\textbf{\textit{x}}_{1}) e_{A}\\
& \times & \sum_{B} e_{B} V_{\textbf{\textit{k}}_{\scriptscriptstyle q}} \left(  F^{B}_{2}(\sqrt{\Vert \textbf{\textit{z}}_{2} \Vert_{c}^{2} + 2 \langle \textbf{\textit{z}}_{2}, . \rangle + \Vert . \Vert_{c}^{2}})\right) (\textbf{\textit{x}}_{2}).
\end{eqnarray*}
\end{proposition}
\begin{proof}
To prove this, we begin by substituting the expression for the CDT of $f$, $\mathcal{F}^{-a,-b} \lbrace f \rbrace (\textbf{\textit{y}})$, into Equation (\ref{EXPLICITD}). Then, by applying Formula (\ref{EXPLICITD2}) along with the Fubini theorem, we derive the following:
\begin{align*}
& \tau_{\textbf{\textit{z}}} f(\textbf{\textit{x}}) \\
& =   \sum_{A} \underbrace{c_{\textbf{\textit{k}}_{p}}\int_{\R^{p}} E_{p, \textbf{\textit{k}}_{p}}(\textbf{\textit{x}}_{1}, a \textbf{\textit{y}}_{1}) E_{p, \textbf{\textit{k}}_{p}}(\textbf{\textit{z}}_{1}, -a \textbf{\textit{y}}_{1}) \overbrace{c_{\textbf{\textit{k}}_{p}} \int_{\R^{p}} E_{p, \textbf{\textit{k}}_{p}}(\textbf{\textit{x}}_{1}, -a \textbf{\textit{y}}_{1}) F^{A}_{1}(\Vert \textbf{\textit{x}}_{1} \Vert_{c}) d\mu^{\textbf{\textit{k}}_{p}}_{p}( \textbf{\textit{x}}_{1})}^{\mathcal{ F}_{\textbf{\textit{k}}_{p}} (F^{A}_{1}(\Vert \textbf{\textit{x}}_{1} \Vert_{c} )) (\textbf{\textit{y}}_{1})} d\mu^{\textbf{\textit{k}}_{p}}_{p}( \textbf{\textit{y}}_{1})}_{\tilde{\tau}^{\textbf{\textit{k}}_{p}}_{\textbf{\textit{z}}_{1}} F^{A}_{1}(\Vert \textbf{\textit{x}}_{1} \Vert_{c})} e_{A} \\
&  \sum_{B} e_{B} \underbrace{c_{\textbf{\textit{k}}_{q}}\int_{\R^{q}} E_{q, \textbf{\textit{k}}_{q}}(\textbf{\textit{x}}_{2}, b \textbf{\textit{y}}_{2}) E_{q, \textbf{\textit{k}}_{q}}(\textbf{\textit{z}}_{2}, -b \textbf{\textit{y}}_{2}) \overbrace{c_{\textbf{\textit{k}}_{q}} \int_{\R^{q}} E_{q, \textbf{\textit{k}}_{q}}(\textbf{\textit{x}}_{2}, -b \textbf{\textit{y}}_{2}) F^{B}_{2}(\Vert \textbf{\textit{x}}_{2} \Vert_{c}) d\mu^{\textbf{\textit{k}}_{q}}_{q}( \textbf{\textit{x}}_{2})}^{\mathcal{F}_{\textbf{\textit{k}}_{q}} (F^{B}_{2}(\Vert \textbf{\textit{x}}_{2} \Vert_{c})) (\textbf{\textit{y}}_{2})} d\mu^{\textbf{\textit{k}}_{q}}_{q}( \textbf{\textit{y}}_{2})}_{\tilde{\tau}^{\textbf{\textit{k}}_{q}}_{\textbf{\textit{z}}_{2}} F^{B}_{2}(\Vert \textbf{\textit{x}}_{2} \Vert_{c})} \\
& =  \sum_{A}  V_{\textbf{\textit{k}}_{p}} \left( F^{A}_{1}(\sqrt{\Vert \textbf{\textit{z}}_{1} \Vert_{c}^{2} + 2 \langle \textbf{\textit{z}}_{1}, . \rangle + \Vert . \Vert_{c}^{2}})\right) (\textbf{\textit{x}}_{1}) e_{A} \\
& \times  \sum_{B} e_{B} V_{\textbf{\textit{k}}_{q}} \left( F^{B}_{2}(\sqrt{\Vert \textbf{\textit{z}}_{2} \Vert_{c}^{2} + 2 \langle \textbf{\textit{z}}_{2}, . \rangle + \Vert . \Vert_{c}^{2}})\right) (\textbf{\textit{x}}_{2}).
\end{align*}

\end{proof}
 \begin{remark}
In light of Remark \ref{hhjjkk}, we can derive an explicit formula for the translation operator associated with the specific hypercomplex Fourier transforms discussed therein. This result follows directly as a corollary of the previous proposition.
\end{remark}
\begin{Ex}
To illustrate the application of the translation operator on a Clifford-valued function, consider the following example:

Let
\begin{align}\label{lll}
f(\textbf{\textit{x}}) = (\alpha + a \beta) e^{- \delta \Vert \textbf{\textit{x}} \Vert_{c}^{2}} (\gamma + b \lambda),
\end{align}
with $\delta > 0$ and ${ \alpha, \, \beta, \, \gamma, \, \lambda \in \R }.$ Then, we have 
\begin{eqnarray*}
\tau_{\textbf{\textit{z}}} f(\textbf{\textit{x}}) &=& ( \alpha + a \beta)(\gamma + b \lambda) V_{\textbf{\textit{k}}_{p}} \left(   e^{- \delta(\Vert \textbf{\textit{z}}_{1} \Vert_{c}^{2} + 2 \langle \textbf{\textit{z}}_{1}, . \rangle + \Vert . \Vert_{c}^{2})} \right) (\textbf{\textit{x}}_{1})\\
& \times &  V_{\textbf{\textit{k}}_{q}} \left(   e^{- \delta(\Vert \textbf{\textit{z}}_{2} \Vert_{c}^{2} + 2 \langle \textbf{\textit{z}}_{2}, . \rangle + \Vert . \Vert_{c}^{2})} \right) (\textbf{\textit{x}}_{2}).
\end{eqnarray*}
\end{Ex}
To further demonstrate the versatility of the generalized translation operator, we focus on the specific case of two dimensions, which provides valuable insights into its behavior in lower-dimensional settings. In this context, we consider functions with quaternionic values defined on $\mathbb{R}^2$. Specifically, let $C( \mathbb{R}^{2}; \mathbb{H})$ denote the set of functions $f : \mathbb{R}^{2} \longrightarrow \mathbb{H}$ such that $f$ can be expressed in the form $f = \sum_{A} f_{A} e_{A}$, where each $f_{i} \in C(\mathbb{R}^{2})$ for $i = 0, \dots, 3$, and where $e_{A} \in \{ e_{0}, e_{1}, e_{2}, e_{1}e_{2} \}$. By defining $\Omega_{i} = \sqrt{ \textbf{\textit{z}}_{i}^{2} + \textbf{\textit{x}}_{i}^{2} - 2 \textbf{\textit{z}}_{i} \textbf{\textit{x}}_{i} t_{i}}$ for $i = 1, 2$, we derive an explicit expression for the generalized translation operator in this particular case.
\begin{proposition}
Let $\textbf{\textit{k}}$ be the unique multiplicity function. For $f \in C(\mathbb{R}^{2}; \mathbb{H})$, the translation $\tau_{\textbf{\textit{z}}}$ has the following explicit form:
\begin{eqnarray*}
\tau_{\textbf{\textit{z}}} f(\textbf{\textit{x}}) & = &  \frac{1}{2} \int_{-1}^{1} g(\Omega_{2}) \left( 1 + \frac{\textbf{\textit{z}}_{2} - \textbf{\textit{x}}_{2}}{\Omega_{2}} \right) \psi_{\textbf{\textit{k}}}(t_{2}) dt_{2} \\
 & + & \frac{1}{2} \int_{-1}^{1} g(-\Omega_{2}) \left( 1 - \frac{\textbf{\textit{z}}_{2} - \textbf{\textit{x}}_{2}}{\Omega_{2}} \right) \psi_{\textbf{\textit{k}}}(t_{2}) dt_{2},
\end{eqnarray*}
where
\begin{eqnarray*}
g(.) & = & \frac{1}{2} \int_{-1}^{1} f(\Omega_{1}, .) \left( 1 + \frac{\textbf{\textit{z}}_{1} - \textbf{\textit{x}}_{1}}{\Omega_{1}} \right) \psi_{\textbf{\textit{k}}}(t_{1}) dt_{1} \\
& + & \frac{1}{2} \int_{-1}^{1} f(-\Omega_{1}, .) \left( 1 - \frac{\textbf{\textit{z}}_{1} - \textbf{\textit{x}}_{1}}{\Omega_{1}} \right) \psi_{\textbf{\textit{k}}}(t_{1}) dt_{1}
\end{eqnarray*}
and
 \begin{align*}
 \psi_{\textbf{\textit{k}}}(t) = \frac{\Gamma (\textbf{\textit{k}} + \frac{1}{2})}{\sqrt{\pi} \Gamma (\textbf{\textit{k}})} (1+t)(1-t^{2})^{\textbf{\textit{k}}-1}.
 \end{align*}
\end{proposition}

\begin{proof}
Applying Formula (\ref{EXPLICITD1}) in two instances yields
\begin{align*}
&\tau_{\textbf{\textit{z}}} f(\textbf{\textit{x}})\\
 &= c_{\textbf{\textit{k}}}^{2} \int_{\mathbb{R}^{2}} \underbrace{c_{\textbf{\textit{k}}} \int_{\mathbb{R}} E_{\textbf{\textit{k}}}(\textbf{\textit{x}}_{1}, a\textbf{\textit{y}}_{1}) E_{\textbf{\textit{k}}}(z_{1}, -a\textbf{\textit{y}}_{1}) \overbrace{c_{\textbf{\textit{k}}} \int_{\mathbb{R}} E_{\textbf{\textit{k}}}(\textbf{\textit{x}}_{1}, -a\textbf{\textit{y}}_{1}) f(\textbf{\textit{x}}_{1}, \textbf{\textit{x}}_{2}) d\mu^{\textbf{\textit{k}}}_{1}( \textbf{\textit{x}}_{1})}^{\mathcal{F}_{\textbf{\textit{k}}}(f(\textbf{\textit{x}}_{1}, \cdot))(\textbf{\textit{y}}_{1})} d\mu^{\textbf{\textit{k}}}_{1}( \textbf{\textit{y}}_{1})}_{\tau_{z_{1}} f(\textbf{\textit{x}}_{1}, \cdot)} \\
&  E_{\textbf{\textit{k}}}(\textbf{\textit{x}}_{2}, -b\textbf{\textit{y}}_{2}) E_{\textbf{\textit{k}}}(z_{2}, -b\textbf{\textit{y}}_{2}) E_{\textbf{\textit{k}}}(\textbf{\textit{x}}_{2}, b\textbf{\textit{y}}_{2}) d\mu^{\textbf{\textit{k}}}_{1}( \textbf{\textit{x}}_{2}) d\mu^{\textbf{\textit{k}}}_{1}( \textbf{\textit{y}}_{2}) \\
&= c_{\textbf{\textit{k}}}^{2} \int_{\mathbb{R}^{2}} g(\textbf{\textit{x}}_{2}) E_{\textbf{\textit{k}}}(\textbf{\textit{x}}_{2}, -b\textbf{\textit{y}}_{2}) E_{\textbf{\textit{k}}}(z_{2}, -b\textbf{\textit{y}}_{2}) E_{\textbf{\textit{k}}}(\textbf{\textit{x}}_{2}, b\textbf{\textit{y}}_{2}) d\mu^{\textbf{\textit{k}}}_{1}( \textbf{\textit{x}}_{2}) d\mu^{\textbf{\textit{k}}}_{1}( \textbf{\textit{y}}_{2}) \\
&= c_{\textbf{\textit{k}}} \int_{\mathbb{R}} c_{\textbf{\textit{k}}} \int_{\mathbb{R}} g(\textbf{\textit{x}}_{2}) E_{\textbf{\textit{k}}}(\textbf{\textit{x}}_{2}, -b\textbf{\textit{y}}_{2}) d\mu^{\textbf{\textit{k}}}_{1}( \textbf{\textit{x}}_{2}) E_{\textbf{\textit{k}}}(z_{2}, -b\textbf{\textit{y}}_{2}) E_{\textbf{\textit{k}}}(\textbf{\textit{x}}_{2}, b\textbf{\textit{y}}_{2}) d\mu^{\textbf{\textit{k}}}_{1}( \textbf{\textit{y}}_{2}) \\
&=  \tau_{z_{2}} g(\textbf{\textit{x}}_{2}).
\end{align*}
\end{proof}
As a consequence, we conclude the results of this section by defining the convolution for the QDT
\begin{definition}
For $ f, \, g \in L^{2}_{\textbf{\textit{k}}_{\scriptscriptstyle d}}(\mathbb{R}^{p,q};Cl_{p,q}) $, the generalized convolution denoted by $f \star g$, is defined by
 \begin{align}\label{CPR}
 (f \star g)( \textbf{\textit{x}} )  =  \int_{\mathbb{R}^{p,q}}  f(\textbf{\textit{z}}) \tau_{\textbf{\textit{z}}} g(\textbf{\textit{x}}) d\mu_{p,q}(\textbf{\textit{z}}).
 \end{align}
\end{definition}
 \section{Analogue of Miyachi's theorem for the C.D.T}
 In this section, we establish an analogue of Miyachi's theorem for the Clifford Dunkl transform. Before presenting the theorem, we introduce the following lemma, which is essential for its proof.
 \begin{lemma} \cite[Lemma 1]{key-co}\label{lem}\\
Let $h$ be an entire function defined on $\mathbb{C}^{d}$, satisfying the following conditions:
$$\| h(y) \| \leq A e^{B \| \operatorname{Re} (z) \|^{2}} $$ 
and 
$$ \int_{\mathbb{R}^{d}} \log^{+} \| h(y) \| \, dy < \infty,$$

for some positive constants $A$ and $B$. Then, $h$ is a constant function.
\end{lemma}

\begin{theorem}\label{miyachi}
Let $f: \mathbb{R}^{p,q} \rightarrow Cl_{p,q}$ be a measurable function such that
\begin{align}\label{condd1}
e^{\alpha \Vert \textbf{\textit{x}} \Vert_{c}^{2}} f \in L^{n}_{\textbf{\textit{k}}_{\scriptscriptstyle d}}(\mathbb{R}^{p,q};Cl_{p,q}) + L^{m}_{\textbf{\textit{k}}_{\scriptscriptstyle d}}(\mathbb{R}^{p,q};Cl_{p,q}),
\end{align}
and
\begin{align}\label{condd2}
\int_{\mathbb{R}^{p,q}} \log^{+} \frac{\Vert \mathcal{F}^{-a,-b}\lbrace f \rbrace(\textbf{\textit{y}})e^{\beta \Vert \textbf{\textit{y}} \Vert_{c}^{2} } \Vert_{c}^{2}}{\lambda} d\textbf{\textit{y}} < \infty,
\end{align}
for some constants $\alpha > 0$, $\beta > 0$, $\lambda > 0$, and $1 \leq n, m \leq +\infty$.
\begin{enumerate}
    \item If $\alpha \beta > \frac{1}{4}$, then $f = 0$ almost everywhere.
    \item If $\alpha \beta = \frac{1}{4}$, then $f(\textbf{\textit{x}}) = C e^{-\alpha \Vert \textbf{\textit{x}} \Vert_{c}^{2}}$ with $\Vert C \Vert_{c} \leq \lambda$.
    \item If $\alpha \beta < \frac{1}{4}$, then for all $\delta \in ]\beta, \frac{1}{4\alpha}[$, all functions of the form $f(\textbf{\textit{x}}) = \frac{1}{(2\delta)^{\gamma_{\scriptscriptstyle d} + d}} P(\textbf{\textit{x}}) e^{- \frac{\Vert \textbf{\textit{x}} \Vert_{c}^{2}}{4\delta}}$, where $P$ is a polynomial taking values in $Cl_{p,q}$, satisfy (\ref{condd1}) and (\ref{condd2}).
\end{enumerate}
\end{theorem}
\begin{proof}\rm
We will divide the proof into two steps.

\textbf{Step $1$}: We will show that the function $\mathcal{ F}^{-a,-b} \lbrace f \rbrace$ is well defined on $\mathbb{C}^{d}$ and there exists $C > 0$ such that 
\begin{align}\label{condd3}
\Vert \mathcal{ F}^{-a,-b} \lbrace f  \rbrace (\textbf{\textit{z}}) \Vert_{c} \leq C e^{\frac{\Vert  \operatorname{Im}(\textbf{\textit{z}}) \Vert_{c}^{2}}{4 \alpha}}, ~~ \forall \textbf{\textit{z}} \in \mathbb{C}^{d}.
\end{align}
Let $ \textbf{\textit{z}} \in \mathbb{C}^{d}$,  by complexifying the variable 
\begin{align}\label{condd4}
\textbf{\textit{z}} = \varepsilon +i_{\C} \eta,
\end{align}
where $ \varepsilon = (\varepsilon_{1}, \varepsilon_{2}), ~~ \eta = ( \eta_{1}, \eta_{2}) \in \mathbb{R}^{p,q} $. 

Then, using (\ref{DK}), we obtain
\begin{eqnarray*}
\Vert \mathcal{ F}^{-a,-b} \lbrace f(\textbf{\textit{x}})  \rbrace (\varepsilon +i_{\C} \eta) \Vert_{c} & \leq & \int_{\mathbb{R}^{p,q}} e^{\Vert  
\textbf{\textit{x}}_{1}\Vert_{c} \Vert \eta_{1} \Vert_{c}} \Vert f(\textbf{\textit{x}}_{1},\textbf{\textit{x}}_{2})\Vert_{c} e^{\Vert  
\textbf{\textit{x}}_{2}\Vert_{c} \Vert \eta_{2} \Vert_{c}} d\mu_{p,q}(\textbf{\textit{x}})\\
& \leq & \int_{\mathbb{R}^{p,q}} e^{\alpha \Vert \textbf{\textit{x}}_{1} \Vert_{c}^{2}} e^{\frac{\Vert \eta_{1} \Vert_{c}^{2}}{4 \alpha}} e^{\Vert  
\textbf{\textit{x}}_{1}\Vert_{c} \Vert \eta_{1} \Vert_{c} - \frac{\Vert \eta_{1} \Vert_{c}^{2}}{4 \alpha} -\alpha \Vert \textbf{\textit{x}}_{1} \Vert_{c}^{2} } \Vert f(\textbf{\textit{x}}_{1},\textbf{\textit{x}}_{2})\Vert_{c}\\
& & e^{\alpha \Vert \textbf{\textit{x}}_{2} \Vert_{c}^{2}} e^{\frac{\Vert \eta_{2} \Vert_{c}^{2}}{4 \alpha}} e^{\Vert  
\textbf{\textit{x}}_{2}\Vert_{c} \Vert \eta_{2} \Vert_{c} - \frac{\Vert \eta_{2} \Vert_{c}^{2}}{4 \alpha} -\alpha \Vert \textbf{\textit{x}}_{2} \Vert_{c}^{2} }  
 d\mu_{p,q}(\textbf{\textit{x}}).
\end{eqnarray*}
We put 
\begin{align*}
g(\textbf{\textit{x}}_{1},\textbf{\textit{x}}_{2}) = e^{\Vert  \textbf{\textit{x}}_{1}\Vert_{c} \Vert \eta_{1} \Vert_{c} - \frac{\Vert \eta_{1} \Vert_{c}^{2}}{4 \alpha} -\alpha \Vert \textbf{\textit{x}}_{1} \Vert_{c}^{2} } e^{\Vert  
\textbf{\textit{x}}_{2}\Vert_{c} \Vert \eta_{2} \Vert_{c} - \frac{\Vert \eta_{2} \Vert_{c}^{2}}{4 \alpha} -\alpha \Vert \textbf{\textit{x}}_{2} \Vert_{c}^{2} }
= e^{-\alpha ( \Vert \textbf{\textit{x}}_{1} \Vert_{c}^{} - \frac{\Vert \eta_{1} \Vert_{c}^{}}{2 \alpha})^{2}} e^{-\alpha ( \Vert \textbf{\textit{x}}_{2} \Vert_{c}^{} - \frac{\Vert \eta_{2} \Vert_{c}^{}}{2 \alpha})^{2}}.
\end{align*}
Then, since $g(\textbf{\textit{x}}_{1},\textbf{\textit{x}}_{2}) \in  L^{n}_{\textbf{\textit{k}}_{\scriptscriptstyle d}}(\mathbb{R}^{p,q};Cl_{p,q}) \cap L^{m}_{\textbf{\textit{k}}_{\scriptscriptstyle d}}(\mathbb{R}^{p,q};Cl_{p,q})$, $\Vert \textbf{\textit{x}} \Vert_{c}^{2} = \Vert \textbf{\textit{x}}_{1} \Vert_{c}^{2}+ \Vert \textbf{\textit{x}}_{2} \Vert_{c}^{2}$ and $ \Vert \eta \Vert_{c}^{2} = \Vert \eta_{1} \Vert_{c}^{2} + \Vert \eta_{2} \Vert_{c}^{2}$, we have
\begin{eqnarray*}
\Vert \mathcal{ F}^{-a,-b} \lbrace f(\textbf{\textit{x}})  \rbrace (\varepsilon +i_{\C} \eta) \Vert_{c}& \leq & e^{\frac{\Vert \eta \Vert_{c}^{2}}{4 \alpha}} \int_{\mathbb{R}^{p,q}} e^{\alpha \Vert \textbf{\textit{x}} \Vert_{c}^{2}} \Vert f(\textbf{\textit{x}})\Vert_{c}  g(\textbf{\textit{x}}) d\mu_{p,q}(\textbf{\textit{x}}).
\end{eqnarray*}
Hence,  $\mathcal{ F}^{-a,-b} \lbrace f \rbrace$  is well defined on $\mathbb{C}^{d}$.\\
On the other hand, by (\ref{condd1}) there exists $ u \in  L^{n}_{\textbf{\textit{k}}_{\scriptscriptstyle d}}(\mathbb{R}^{p,q};Cl_{p,q})$ and  $ v \in  L^{m}_{\textbf{\textit{k}}_{\scriptscriptstyle d}}(\mathbb{R}^{p,q};Cl_{p,q})$ such that 
\begin{align*}
f(\textbf{\textit{x}}) =e^{-\alpha \Vert \textbf{\textit{x}} \Vert_{c}^{2}} (u(\textbf{\textit{x}})+ v(\textbf{\textit{x}})).
\end{align*}
Using triangle inequality, the linearity of the integral, and H\"{o}lder's inequality, one has
\begin{eqnarray*}
\Vert \mathcal{ F}^{-a,-b} \lbrace f(\textbf{\textit{x}})  \rbrace (\varepsilon +i_{\C} \eta) \Vert_{c} & \leq & e^{\frac{\Vert \eta \Vert_{c}^{2}}{4 \alpha}} \left(  \int_{\mathbb{R}^{p,q}} \Vert u(\textbf{\textit{x}}) \Vert_{c} g(\textbf{\textit{x}}) d\mu_{p,q}(\textbf{\textit{x}}) +\int_{\mathbb{R}^{p,q}} \Vert v(\textbf{\textit{x}}) \Vert_{c}  g(\textbf{\textit{x}}) d\mu_{p,q}(\textbf{\textit{x}})\right).\\
& \leq & e^{\frac{\Vert \eta \Vert_{c}^{2}}{4 \alpha}} \left( \vert u(\textbf{\textit{x}}) \vert_{\textbf{\textit{k}}_{\scriptscriptstyle d},n}\vert g(\textbf{\textit{x}}) \vert_{\textbf{\textit{k}}_{\scriptscriptstyle d},m} +\vert v(\textbf{\textit{x}}) \vert_{\textbf{\textit{k}}_{\scriptscriptstyle d},m}\vert g(\textbf{\textit{x}}) \vert_{\textbf{\textit{k}}_{\scriptscriptstyle d},n} \right).
\end{eqnarray*}
Therefore, the desired result follows.

\textbf{Step $2$:} Let 
\begin{align*}
h(\textbf{\textit{z}}) = e^{\frac{-\textbf{\textit{z}}^{2}}{4\alpha}} \mathcal{ F}^{-a,-b} \lbrace f  \rbrace (\textbf{\textit{z}}), ~~ \forall \textbf{\textit{z}} \in \mathbb{C}^{d}.
\end{align*}
Clearly, $h$ is an entire function.\\ 
Using (\ref{condd4}), we have 
\begin{align*}
\textbf{\textit{z}}^{2} = - \varepsilon^{2} + \eta^{2}-2i_{\C} \langle \varepsilon ,\eta \rangle,
\end{align*}
which implies
\begin{align*}
\Vert e^{\frac{-\textbf{\textit{z}}^{2}}{4\alpha}} \Vert_{c} \leq e^{\frac{\Vert \varepsilon \Vert_{c}^{2}}{4\alpha}} e^{-\frac{\Vert \eta \Vert_{c}^{2}}{4\alpha}}.
\end{align*}
Formula (\ref{condd3}) yields
\begin{align}\label{hc}
\Vert h(\textbf{\textit{z}}) \Vert_{c} \leq C e^{\frac{\Vert \varepsilon \Vert_{c}^{2}}{4\alpha}}.
\end{align}
\begin{enumerate}
\item If $ \alpha \beta  > \frac{1}{4}$, hence $\log^{+}(cs) \leq \log^{+}(c)+s$ for all $c, \, s >0$, it follows that:

\begin{eqnarray*}
\int_{\mathbb{R}^{p,q}} \log^{+} \Vert h(\textbf{\textit{y}}) \Vert_{c} d\textbf{\textit{y}} & = & \int_{\mathbb{R}^{p,q}} \log^{+} \Vert e^{\frac{-\textbf{\textit{y}}^{2}}{4\alpha}} \mathcal{ F}^{-a,-b} \lbrace f  \rbrace (\textbf{\textit{y}})\Vert_{c} d\textbf{\textit{y}}\\
& = & \int_{\mathbb{R}^{p,q}} \log^{+} \Vert e^{\frac{ \Vert \textbf{\textit{y}} \Vert_{c}^{2}}{4\alpha}} \mathcal{ F}^{-a,-b} \lbrace f  \rbrace (\textbf{\textit{y}}) \Vert_{c} d\textbf{\textit{y}}\\
& =& \int_{\mathbb{R}^{p,q}} \log^{+} \left(  \frac{\Vert e^{\beta \Vert \textbf{\textit{y}} \Vert_{c}^{2} }\mathcal{ F}^{-a,-b} \lbrace f  \rbrace (\textbf{\textit{y}})\Vert_{c}}{\lambda} \lambda e^{(\frac{1}{4\alpha}-\beta) \Vert \textbf{\textit{y}} \Vert_{c}^{2} } \right)  d\textbf{\textit{y}}\\
& \leq & \int_{\mathbb{R}^{p,q}} \log^{+}  \frac{\Vert e^{\beta \Vert \textbf{\textit{y}} \Vert_{c}^{2} }\mathcal{ F}^{-a,-b} \lbrace f  \rbrace (\textbf{\textit{y}})\Vert_{c}}{\lambda} d\textbf{\textit{y}} + \int_{\mathbb{R}^{p,q}}  \lambda e^{(\frac{1}{4\alpha}-\beta) \Vert \textbf{\textit{y}} \Vert_{c}^{2} } d\textbf{\textit{y}}.
\end{eqnarray*}
Using the fact that $\alpha \beta > \frac{1}{4}$ and the assumption (\ref{condd2}), we obtain
\begin{align*}
\int_{\mathbb{R}^{p,q}} \log^{+} \Vert h(\textbf{\textit{y}}) \Vert_{c} d\textbf{\textit{y}}  <  +\infty.
\end{align*}
Thus, $h$ satisfies the conditions of Lemma \ref{lem}; therefore, $h$ is constant.
\begin{align*}
\mathcal{ F}^{-a,-b} \lbrace f  \rbrace (\textbf{\textit{y}}) = C e^{-\frac{ \Vert \textbf{\textit{y}} \Vert_{c}^{2}}{4\alpha}}.
\end{align*}
According to the relation (\ref{condd2}), the constant $C$ is zero.

Then, $\mathcal{ F}^{-a,-b} \lbrace f  \rbrace (\textbf{\textit{y}}) = 0$, which implies that $f = 0$.
\item If $ \alpha \beta  = \frac{1}{4}$, employing the same proof as in the previous case, we obtain
\begin{align*}
 \int_{\mathbb{R}^{p,q}} \log^{+} \frac{\Vert h(\textbf{\textit{y}}) \Vert_{c}}{\lambda} d\textbf{\textit{y}} =  \int_{\mathbb{R}^{p,q}} \log^{+} \frac{\Vert e^{\beta \Vert \textbf{\textit{y}} \Vert_{c}^{2}} \mathcal{ F}^{-a,-b} \lbrace f  \rbrace (\textbf{\textit{y}}) \Vert_{c}}{\lambda} d\textbf{\textit{y}} <  +\infty.
\end{align*}
Then, from (\ref{hc}) and Lemma \ref{lem}, we deduce that 
 \begin{align*}
 \mathcal{ F}^{-a,-b} \lbrace f  \rbrace (\textbf{\textit{y}}) = C e^{- \frac{ \Vert \textbf{\textit{y}} \Vert_{c}^{2}}{4 \alpha }}.
 \end{align*}
 So, the assumption on $\mathcal{ F}^{-a,-b} \lbrace f  \rbrace $, implies that $ \Vert C \Vert_{c} \leq \lambda $.\\
Then, by using inversion formulas (\ref{inversin}) and (\ref{ll66}), we have
\begin{eqnarray*}
f (\textbf{\textit{x}}) & = & c^{2}_{\textbf{\textit{k}}_{p}} c^{2}_{\textbf{\textit{k}}_{q}} \int_{\mathbb{R}^{p,q}} E_{p, \textbf{\textit{k}}_{p}}(\textbf{\textit{x}}_{1},a \textbf{\textit{y}}_{1}) \mathcal{ F}^{-a,-b} \lbrace f (\textbf{\textit{x}}_{1},\textbf{\textit{x}}_{2}) \rbrace(\textbf{\textit{y}}_{1},\textbf{\textit{y}}_{2})E_{q, \textbf{\textit{k}}_{q}}(\textbf{\textit{x}}_{2},b \textbf{\textit{y}}_{2})d\mu_{p,q}(\textbf{\textit{y}})\\
& = & c^{2}_{\textbf{\textit{k}}_{p}} c^{2}_{\textbf{\textit{k}}_{q}} \int_{\mathbb{R}^{p,q}} E_{p, \textbf{\textit{k}}_{p}}(\textbf{\textit{x}}_{1},a \textbf{\textit{y}}_{1})C  e^{- \frac{ \Vert \textbf{\textit{y}} \Vert_{c}^{2}}{4 \alpha }} E_{q, \textbf{\textit{k}}_{q}}(\textbf{\textit{x}}_{2},b \textbf{\textit{y}}_{2})d\mu_{p,q}(\textbf{\textit{y}})\\
&= &  c^{2}_{\textbf{\textit{k}}_{p}} c^{2}_{\textbf{\textit{k}}_{q}}  \int_{\mathbb{R}^{p}} E_{p,\textbf{\textit{k}}_{p}}(\textbf{\textit{x}}_{1},a \textbf{\textit{y}}_{1})C  e^{- \frac{ \Vert \textbf{\textit{y}}_{1} \Vert_{c}^{2}}{4 \alpha }} d\mu^{\textbf{\textit{k}}_{p}}_{p}(\textbf{\textit{y}}_{1}) C  \int_{\mathbb{R}^{q}}  e^{- \frac{ \Vert \textbf{\textit{y}}_{2} \Vert_{c}^{2}}{4 \alpha }} E_{q, \textbf{\textit{k}}_{q}}(\textbf{\textit{x}}_{2},b \textbf{\textit{y}}_{2}) d\mu^{\textbf{\textit{k}}_{q}}_{q}(\textbf{\textit{y}}_{2})\\
&=& c^{2}_{\textbf{\textit{k}}_{p}} c^{2}_{\textbf{\textit{k}}_{q}}  (2 \alpha )^{-(\gamma_{p} +\frac{p}{2})} e^{- \alpha \Vert \textbf{\textit{x}}_{1} \Vert_{c}^{2} } C (2 \alpha )^{-(\gamma_{q} +\frac{q}{2})} e^{- \alpha \Vert \textbf{\textit{x}}_{2} \Vert_{c}^{2} }\\
&=& C c^{2}_{\textbf{\textit{k}}_{p}} c^{2}_{\textbf{\textit{k}}_{q}}  (2 \alpha )^{-(\gamma_{d} +\frac{d}{2})} e^{- \alpha \Vert \textbf{\textit{x}} \Vert_{c}^{2}}.
\end{eqnarray*}
 \item If $ \alpha \beta  < \frac{1}{4}$. Let 
 \begin{eqnarray*}
 f(\textbf{\textit{x}})  =  \frac{1}{(2\delta)^{\gamma_{d} +d}}P(\textbf{\textit{x}})e^{- \frac{ \Vert \textbf{\textit{x}} \Vert_{c}^{2}}{4\delta}},
 \end{eqnarray*}
 where $P(\textbf{\textit{x}}) = \sum_{A} P_{A}(\textbf{\textit{x}}_{1},\textbf{\textit{x}}_{2}) e_{A}$, with $ \delta \in ]\beta, \frac{1}{4\alpha}[$, and the polynomials  $P_{A}$  satisfy the assumptions of Miyachi's theorem for the Dunkl transform.
Then, 
 \begin{eqnarray*} 
 & & \mathcal{ F}^{-a,-b} \lbrace \frac{1}{(2\delta)^{\gamma_{d} +d}}P(\textbf{\textit{x}})e^{- \frac{ \Vert \textbf{\textit{x}} \Vert_{c}^{2}}{4\delta}}  \rbrace (\textbf{\textit{y}})\\
 & = &  \int_{\mathbb{R}^{p,q}} E_{p, \textbf{\textit{k}}_{p}}(\textbf{\textit{x}}_{1},-a \textbf{\textit{y}}_{1})\frac{1}{(2\delta)^{\gamma_{d} +d}}P(\textbf{\textit{x}})e^{- \frac{ \Vert \textbf{\textit{x}} \Vert_{c}^{2}}{4\delta}} E_{q, \textbf{\textit{k}}_{q}}(\textbf{\textit{x}}_{2},-b \textbf{\textit{y}}_{2})d\mu_{p,q}(\textbf{\textit{x}})\\
 & =& \sum_{A} \int_{\mathbb{R}^{q}} \left(  \int_{\mathbb{R}^{p}} E_{p, \textbf{\textit{k}}_{p}}(\textbf{\textit{x}}_{1},-a \textbf{\textit{y}}_{1}) \frac{1}{(2\delta)^{\gamma_{p} +p}}P_{A}(\textbf{\textit{x}}_{1}, \textbf{\textit{x}}_{2}) e^{- \frac{ \Vert \textbf{\textit{x}}_{1} \Vert_{c}^{2}}{4\delta}} d\mu^{\textbf{\textit{k}}_{p}}_{p}(\textbf{\textit{x}}_{1}) \right) e_{A} \\ 
& & \frac{1}{(2\delta)^{\gamma_{q} +q}} e^{- \frac{ \Vert \textbf{\textit{x}}_{2} \Vert_{c}^{2}}{4\delta}}  E_{q, \textbf{\textit{k}}_{q}}(\textbf{\textit{x}}_{2},-b \textbf{\textit{y}}_{2}) d\mu^{\textbf{\textit{k}}_{q}}_{q}(\textbf{\textit{x}}_{2}) \\
& =& \sum_{A} \frac{1}{(2\delta)^{\gamma_{q} +q}}  e^{- \delta \Vert \textbf{\textit{y}}_{1} \Vert_{c}^{2}}e_{A} \int_{\mathbb{R}^{q}} Q_{A}(\textbf{\textit{y}}_{1}, \textbf{\textit{x}}_{2})e^{- \frac{ \Vert \textbf{\textit{x}}_{2} \Vert_{c}^{2}}{4\delta}} E_{q, \textbf{\textit{k}}_{q}}(\textbf{\textit{x}}_{2},-b \textbf{\textit{y}}_{2}) d\mu^{\textbf{\textit{k}}_{q}}_{q}(\textbf{\textit{x}}_{2}) \\
& = & \sum_{A} e^{- \delta \Vert \textbf{\textit{y}}_{1} \Vert_{c}^{2}} e^{- \delta \Vert \textbf{\textit{y}}_{2} \Vert_{c}^{2}}Q_{A}(\textbf{\textit{y}}_{1}, \textbf{\textit{y}}_{2})e_{A}\\
& =& e^{- \delta \Vert \textbf{\textit{y}} \Vert_{c}^{2}} Q(\textbf{\textit{y}}_{1}, \textbf{\textit{y}}_{2}),
 \end{eqnarray*}
 where $Q$ is a polynomial of the same degree as $P$.
On the other hand, since $ \delta < \frac{1}{4\alpha}$, we get
\begin{align*}
e^{\alpha \Vert \textbf{\textit{x}} \Vert_{c}^{2}} f(\textbf{\textit{x}} )=\frac{1}{(2\delta)^{\gamma +d}}P(\textbf{\textit{x}}) e^{(\alpha - \frac{1}{4\delta})\Vert \textbf{\textit{x}} \Vert_{c}^{2}} \in L^{n}_{\textbf{\textit{k}}_{\scriptscriptstyle d}}(\mathbb{R}^{p,q};Cl_{p,q}) + L^{m}_{\textbf{\textit{k}}_{\scriptscriptstyle d}}(\mathbb{R}^{p,q};Cl_{p,q}),
\end{align*}
and as $ \delta  > \beta$, it follows that
\begin{align*}
\int_{\mathbb{R}^{p,q}} \log^{+} \frac{\Vert \mathcal{ F}^{-a,-b} \lbrace f \rbrace (\textbf{\textit{y}})e^{\beta \Vert \textbf{\textit{y}} \Vert_{c}^{2} }\Vert_{c}^{2}}{\lambda}d \textbf{\textit{y}} = \int_{\mathbb{R}^{p,q}} \log^{+} \frac{\Vert e^{(\beta - \delta) \Vert \textbf{\textit{y}} \Vert_{c}^{2} } Q(\textbf{\textit{y}}_{1}, \textbf{\textit{y}}_{2})\Vert_{c}^{2}}{\lambda}d \textbf{\textit{y}} < \infty.
\end{align*}
\end{enumerate}
Therefore, the theorem is now proven.
\end{proof}\rm
As a result, we obtain an analogue of the Miyachi theorem for the general double-sided orthogonal planes split quaternion Fourier transform \cite{key-7}.

Indeed, in the case where $p=0$, $q=2$ and $\textbf{\textit{k}}_{\scriptscriptstyle p}=\textbf{\textit{k}}_{\scriptscriptstyle q} = 0$, we have $Cl_{p,q}= Cl_{0,2} \cong \mathbb{H}$, $E_{p, \textbf{\textit{k}}_{\scriptscriptstyle p}}(x_{1},-a y_{1}) = e^{- a  x_{1}y_{1} }$, $E_{q, \textbf{\textit{k}}_{\scriptscriptstyle q}}(x_{2},-b y_{2}) = e^{- b  x_{2}y_{2} }$ and the measure $d\mu_{p,q}$ coincides with the Lebesgue measure on $\mathbb{R}^{2}$. So,
\begin{align}\label{QFT}
\mathcal{ F}^{-a,-b} \lbrace f \rbrace (\textbf{\textit{y}})=  \int_{\mathbb{R}^{2}} e^{-ax_{1}y_{1}}f(\textbf{\textit{x}})e^{-bx_{2}y_{2}}dx_{1}dx_{2}.
\end{align}
Note that this transform is a generalization of the double-sided orthogonal planes split quaternion Fourier transform, as the conditions on $a$ and $b$ being pure quaternions are not required. Thus, the Theorem \ref{miyachi} holds for this transform and is stated as follows:
\begin{Th}
 Let  $ f: \mathbb{R}^{2} \rightarrow \mathbb{H} $ be a measurable function such that
 \begin{align}\label{condd1c}
e^{\alpha \parallel \textbf{\textit{x}} \parallel_{c}^{2}} f \in L^{n}_{0}(\mathbb{R}^{2};\mathbb{H}) + L^{m}_{0}(\mathbb{R}^{2};\mathbb{H}),
\end{align}
and 
\begin{align}\label{condd2c}
\int_{\mathbb{R}^{2}} log^{+} \frac{\parallel \mathcal{ F}^{-a,-b}\lbrace f \rbrace (\textbf{\textit{y}})e^{\beta \parallel \textbf{\textit{y}} \parallel_{c}^{2} }\parallel_{c}^{2}}{\lambda}d \textbf{\textit{y}} < \infty ,
\end{align}
for some constants $\alpha >0$,  $\beta > 0$, $\lambda > 0, $ and $1 \leq n, m \leq +\infty$.
\begin{enumerate}
\item If $ \alpha \beta  > \frac{1}{4}$, then $f=0$ almost everywhere.
\item If  $ \alpha \beta =\frac{1}{4}$, then $f(\textbf{\textit{x}}) =C e^{ \frac{- \parallel \textbf{\textit{x}} \parallel_{c}^{2}}{4\alpha}}$ with $ \mid C \mid \leq \lambda $.
\item If $\alpha \beta < \frac{1}{4}$ ,  then for all $ \delta \in ]\beta, \frac{1}{4\alpha}[$, all functions of the form $ f(\textbf{\textit{x}}) =  \frac{1}{(2\delta)^{2}}P(\textbf{\textit{x}})e^{\frac{- \parallel \textbf{\textit{x}} \parallel_{c}^{2}}{4\delta}}$, where $P$ is a polynomial taking values in $\mathbb{H}$  satisfy (\ref{condd1c}) and (\ref{condd2c}).
\end{enumerate}
 \end{Th}
 \begin{remark}
 If we put $a=\textbf{\textit{i}}$ and $b=\textbf{\textit{j}}$ in the previous theorem, we reobtain Theorem $4.2$ in \cite{elh3}.
 \end{remark}

\section{ Conclusion}
In this paper, we have introduced a new generalization of the Fourier transform in the context of Clifford algebras and reflection groups. We studied some important properties such as linearity, the inversion formula, and Plancherel's theorems. As an application, we then used some of these properties to prove Miyachi's theorem for this transform.

A well-known example of a CDT with two square roots of $-1$ are the quaternion Fourier Transforms(QFTs) \cite{QFT3,  QFT2, key-777, QFA1}, which are particularly used in applications in image processing and signal analysis. Therefore, we expect that the CDT and other generalizations of QFTs obtained by the same methods will be of great potential utility and find applications  in the fields of image processing and signal analysis.\\

%\textbf{ je veux que tu verifier ces etapes aussi:   Faites des transitions fluides entre les sections, Vérifiez la grammaire, l'orthographe et le style, Utilisez un langage formel et précis, Reformulez les phrases pour éviter le plagiat tout en conservant le sens original, Reformulez les phrases pour éviter le plagiat tout en conservant le sens original, Assurez-vous que chaque section et paragraphe a un objectif clair and Expliquez les concepts complexes avec des exemples ou des explications supplémentaires.}
%%%%%%%%%%%%%%%%%%%%%%%%%%%%%%%%%%%%%%%% DECLARATIONS %%%%%%%%%%%%%%%%%%%%%%%%%%%%%%%%%%%

\textbf{ Data Availability } No data were used to support this study.\\
 
\textbf{ Conflicts of interests} On behalf of all authors, the corresponding author states that there is no conflict of interest.\\

\textbf{ Funding statement}  This research received no specific grant from any funding agency in the public, commercial, or not-for-profit sectors.\\

\textbf{Declaration of generative AI and AI-assisted technologies in the writing process}

During the preparation of this manuscript, the authors used ChatGPT (OpenAI, free version, September 2025) to improve the clarity and grammar of the English language. This tool was employed because the authors are non-native English speakers and sought to ensure the linguistic accuracy of the manuscript. The authors carefully reviewed and edited all content and take full responsibility for the final version of the manuscript.

%%%%%%%%%%%%%%%%%%%%%%%%%%%%%%%%%%%%% REFERENCES %%%%%%%%%%%%%%%%%%%%%%%%%%%%%%%%%


\begin{thebibliography}{99}

\bibitem{key-BM} D. Brennecken and M. Rösler, \textit{The Dunkl–Laplace transform and Macdonald’s hypergeometric series}, Trans. Am. Math. Soc. \textbf{376}, 2419–2447 (2023).  


\bibitem{key-BR} F. Brackx, R. Delanghe, and F. Sommen, \textit{Clifford Analysis}, Pitman Publishers, Boston-London-Melbourne, 1982.

\bibitem{key-BR2} F. Brackx, J. S. R. Chisholm, and J. Bures (Eds), \textit{Clifford Analysis and Its Applications}, NATO Science Series, Kluwer, Dordrecht, 2001.

\bibitem{QFT3} T. B\"{u}low, \textit{Hypercomplex Spectral Signal Representations for the Processing and Analysis of Images} [PhD Thesis], University of Kiel, Germany, 1999.

\bibitem{OPD} P. Cerejeiras, U. Kahler, and G. Ren, \textit{Clifford Analysis for Finite Groups}, Complex Variables, \textbf{487-496}, 2006.

\bibitem{key-co} F. Chouchene, R. Daher, T. Kawazoe, and H. Mejjaoli, \textit{Miyachi's Theorem for the Dunkl Transform}, Integral Transform. Spec. Funct, \textbf{22}, 167-173, 2011.

\bibitem{DBY1} H. De Bie, \textit{New techniques for the two-sided quaternionic Fourier transform}, Proceedings of AGACSE, 2012.

\bibitem{DBY} H. De Bie, N. De Schepper, \textit{Clifford-Gegenbauer Polynomials Related to the Dunkl Dirac Operator}, Bull. Belg. Math. Soc. Simon Stevin, \textbf{18}(2), 193-214, 2011. \url{https://doi.org/10.36045/bbms/1307452070}

\bibitem{QFT2} H. De Bie, N. De Schepper, T. A. Ell, K. Rubrecht, and S. J. Sangwine, \textit{Connecting Spatial and Frequency Domains for the Quaternion Fourier Transform}, Appl. Math. Comput, \textbf{271}, 581-593, 2015.

\bibitem{Dor} 
C. Doran, A. Lasenby, \textit{Geometric Algebra for Physicists}, Cambridge University Press, 2003.

\bibitem{key-3} M. F. E. de Jeu, \textit{The Dunkl Transform}, Invent. Math, \textbf{113}, 147-162, 1993.


\bibitem{4444} C. F. Dunkl, \textit{Differential-Difference Operators Associated to Reflection Groups}, Trans. Am. Math. Soc, \textbf{311}, 167-183, 1989.

\bibitem{key-5} C. F. Dunkl, \textit{Hankel Transforms Associated to Finite Reflection Groups}, In: Proc. of the Special Session on Hypergeometric Functions on Domains of Positivity, Jack Polynomials and Applications, Proceedings, Tampa, Contemp. Math, \textbf{138}, 123-138, 1992.

\bibitem{key-6} C. F. Dunkl, \textit{Integral Kernels with Reflection Group Invariance}, Canad. J. Math, \textbf{43}, 1213-1227, 1991.

\bibitem{elh3} Y. El Haoui, S. Fahlaoui, \textit{Miyachi’s Theorem for the Quaternion Fourier Transform}, Circuits Syst. Signal Process, \textbf{39}, 1-14, 2019. \url{https://doi.org/10.1007/s00034-019-01243-6}

\bibitem{key-7} T. A. Ell, \textit{Quaternion-Fourier Transforms for Analysis of 2-Dimensional Linear Time-Invariant Partial-Differential Systems}, In: Proceedings of the 32nd Conference on Decision and Control, San Antonio, pp. 1830-1841, 1993.

\bibitem{key-fah} M. Essenhajy, S. Fahlaoui, \textit{The Two-Sided Quaternionic Dunkl Transform and Hardy’s Theorem}, Rend. Circ. Mat. Palermo II Ser, 2022. \url{https://doi.org/10.1007/s12215-021-00708-5}

\bibitem{key-uncertainty} M. Essenhajy, S. Fahlaoui, \textit{Uncertainty Principles for the Two-Sided Quaternionic Dunkl Transform}, Bol. Soc. Mat. Mex., 2025. \url{https://doi.org/10.1007/s40590-025-00764-2}


\bibitem{GHH} J. Gilbert, M. Murray, \textit{Clifford Algebra and Dirac Operators in Harmonic Analysis}, Cambridge University Press, Cambridge, 1991.

\bibitem{key-HITZ} E. Hitzer, \textit{Two-Sided Clifford Fourier Transform with Two Square Roots of $-1$ in $Cl(p, q)$}, Adv. Appl. Clifford Algebras, \textbf{24}, 313-332, 2014. \url{https://doi.org/10.1007/s00006-014-0441-9}

\bibitem{key-777} E. Hitzer, S. J. Sangwine, \textit{The Orthogonal 2D Planes Split of Quaternions and Steerable Quaternion Fourier Transformations}, In: Hitzer, E., Sangwine, S. (eds) Quaternion and Clifford Fourier Transforms and Wavelets, Trends in Mathematics, Birkh\"{a}user, Basel, 2013. \url{https://doi.org/10.1007/978-3-0348-0603-9_2}


\bibitem{key-MF2} H. Monaim and S. Fahlaoui, \textit{General one-dimensional Clifford Fourier Transform and applications to probability theory}, Rend. Circ. Mat. Palermo, II. Ser. \textbf{73}, 1453–1466 (2024). \url{https://doi.org/10.1007/s12215-023-00994-1}

\bibitem{key-MF1} H. Monaim and S. Fahlaoui, \textit{General Right-Sided Orthogonal 2D-Planes Split Quaternionic Wave-Packet Transform}, Adv. Appl. Clifford Algebras \textbf{33}, 58 (2023). \url{https://doi.org/10.1007/s00006-023-01303-w}

\bibitem{Por} 
R. M. Porter, I. R. Porteous, \textit{Clifford Algebras and Analysis: History and Applications}, Oxford University Press, 1995.


%\bibitem{key-RTT} M. R\"{o}sler, M. Voit, \textit{An Uncertainty Principle for Hankel Transforms}, Proc. Amer. Math. Soc, \textbf{127}(1), 183-194, 1999.

\bibitem{rosler4} M. R\"{o}sler, \textit{A positive radial product formula for the Dunkl kernel}, Trans. Amer. Math. Soc., \textbf{355}(6), 2413-2438, 2003.

\bibitem{rosler6} M. R\"{o}sler, \textit{Bessel-type signed hypergroups on $\mathbb{R}$. In Probability measures on groups and related structures}, XI (Oberwolfach, 1994), pages 292-304, World Sci. Publ., River Edge, NJ, 1995.


\bibitem{co} M. R\"{o}sler, \textit{Generalized Hermite Polynomials and the Heat Equation for Dunkl Operators}, Commun. Math. Phys, \textbf{192}, 519-541, 1998.

\bibitem{MR} M. R\"{o}sler, \textit{Dunkl Operators: Theory and Applications}, In: Orthogonal Polynomials and Special Functions (Leuven, 2002), Lect. Notes Math, \textbf{1817}, Springer-Verlag, 93-135, 2003.

\bibitem{key-TD3} F. Saadi, O. Tyr, and R. Daher, \textit{Generalized Hausdorff Operators on $\dot{K}^{\alpha}$}, J. Funct. Spaces \textbf{2021} (2021). \url{https://api.semanticscholar.org/CorpusID:239672350}

\bibitem{QFA1} S. J. Sangwine, \textit{Fourier Transforms of Colour Images Using Quaternion, or Hypercomplex, Numbers}, Electronics Letters, \textbf{32}(21), 1996.


\bibitem{shi} H. Shi, H. Yang, Z. Li, \textit{Two-Sided Fourier Transform in Clifford Analysis and Its Application}, Adv. Appl. Clifford Algebras, \textbf{30}, 67, 2020. \url{https://doi.org/10.1007/s00006-020-01083-7}

\bibitem{phy} W. Spr\"{o}ssig, \textit{Clifford Analysis and Its Applications in Mathematical Physics}, Cubo Matem\'{a}tica Educacional, \textbf{4}, 253-314, 2002.

\bibitem{TX} S. Thangavelu, Y. Xu, \textit{Convolution Operator and Maximal Function for the Dunkl Transform}, J. Anal. Math, \textbf{97}, 25-55, 2005.

\bibitem{key-TD1} O. Tyr and R. Daher, \textit{Benedicks–Amrein–Berthier type theorem and local uncertainty principles in Clifford algebras}, Rend. Circ. Mat. Palermo, II. Ser. \textbf{72}, 99–115 (2023). \url{https://doi.org/10.1007/s12215-021-00669-9}

\bibitem{key-TD2} O. Tyr and R. Daher, \textit{Beurling’s Theorem in the Clifford Algebras}, Adv. Appl. Clifford Algebras \textbf{33}, 37 (2023). \url{https://doi.org/10.1007/s00006-023-01284-w}

\bibitem{key-TD5} O. Tyr and R. Daher, \textit{Some direct and inverse theorems of approximation of functions in Jacobi-Dunkl discrete harmonic analysis}, J. Math. Sci. \textbf{266}, 534–553, 2022. \url{https://doi.org/10.1007/s10958-022-05891-z}

\bibitem{key-TD4} O. Tyr, F. Saadi, and R. Daher, \textit{On the generalized Hilbert transform and weighted Hardy spaces in q-Dunkl harmonic analysis}, Ramanujan J. \textbf{60}, 95–122, 2023. \url{https://doi.org/10.1007/s11139-022-00666-1}



\bibitem{DAP1} J. F. van Diejen, L. Vinet, \textit{Calogero-Sutherland-Moser Models}, CRM Series in Math. Phys, Springer-Verlag, 2000.

\bibitem{jv} J. Vince, \textit{Geometric Algebra for Computer Graphics}, Springer Science and Business Media, 2008.

\end{thebibliography}
\end{document}